\documentclass[11pt]{amsart}

\usepackage{amsmath,amssymb,amsfonts,amsthm,enumerate}
\usepackage{vmargin}
\usepackage[dvips]{graphicx}
\usepackage{color}
\input{amssym.def}
\input{amssym}
\input{epsf}
\usepackage{supertabular}
\usepackage{array}
\usepackage{multicol}
\usepackage{ulem}

\usepackage{paralist}
\usepackage{graphicx}
\usepackage[colorlinks=true]{hyperref}

\usepackage[titletoc]{appendix}

\def\R{\mathbb R}
\def\N{\mathbb N}

\def\hat{\widehat}

\def\F{{\mathcal F}}

\renewcommand{\bar}{\overline}

\def\H{\mathcal H}

\def\sign{\mathrm{sign}\,}

\def\In{\mathcal{I}}
\def\Inl{\mathcal{J}}
\def\Inb{\partial{\mathcal{J}}}

\def\Out{\mathcal{O}}
\def\Outl{\mathcal{E}}
\def\V{\mathcal{V}}
\def\Yinf{\bar Y}

\def\<{\langle}
\def\>{\rangle}
\def\Rord{\mathcal{R}}
\def\W{\mathcal{W}}
\def\U{\mathcal{U}}

\newtheorem{Thm}{Theorem}

\newtheorem{Lem}[Thm]{Lemma}

\newtheorem{Cor}[Thm]{Corollary}
\newtheorem{Prop}[Thm]{Proposition}

\numberwithin{equation}{section}
\numberwithin{Thm}{section}

\theoremstyle{definition}
\newtheorem{Def}[Thm]{Definition}
\newtheorem{Rk}[Thm]{Remark}

\theoremstyle{remark}

\newtheorem*{Thm*}{Theorem}
\newtheorem*{Lem*}{Lemma}
\newtheorem*{Conj*}{Conjecture}
\newtheorem*{Cor*}{Corollary}
\newtheorem*{Def*}{Definition}
\newtheorem*{Prop*}{Proposition}
\newtheorem*{Exo*}{Exercise}
\newtheorem*{Exs*}{Examples}
\newtheorem*{Ex*}{Example}
\newtheorem*{Rk*}{Remark}
\newtheorem*{Rks*}{Remarks}

\author{Vincent Calvez}
\address{Unit\'e de Math\'ematiques Pures et
Appliqu\'ees, Ecole Normale Sup\'erieure de Lyon, CNRS UMR 5669, and
 project-team Inria NUMED, Lyon, France. E-mail: {\tt
 vincent.calvez@ens-lyon.fr}}
 
\author{Thomas Gallou\"et}
\address{Thomas O. Gallou\"et: CMLS, \'Ecole polytechnique, CNRS, Universit\'e Paris-Saclay, 91128 Palaiseau Cedex, France. E-mail: {\tt
 thomas.gallouet@polytechnique.edu}}



\begin{document}
\title[Blow-up of homogeneous gradient flows]{
Blow-up phenomena for  gradient flows of discrete homogeneous functionals}

\begin{abstract}
We investigate gradient flows of some homogeneous functionals in $\R^N$, arising in the Lagrangian approximation of systems of self-interacting and diffusing particles. We focus on the case of negative homogeneity. In the case of strong self-interaction, the functional  possesses a cone of negative energy. It is immediate to see that solutions with negative energy at some time become singular in finite time, meaning that a subset of particles concentrate at a single point. Here, we establish that all solutions become singular in finite time for the class of functionals under consideration. The paper is completed with numerical simulations illustrating the striking non linear dynamics when initial data have positive energy.
\end{abstract}

\maketitle

\section{Introduction and main result}


We investigate a deterministic particle approximation of the gradient flow of homogeneous functionals in the family
\begin{equation}\mathcal G_m[\rho]  = \frac{1}{m-1} \int_{\R} \rho (x)^m \, dx - \frac{\chi}{m-1} \iint_{\R\times \R} |x-y|^{1-m}\rho(x)\rho(y)\, dx dy\, , \label{eq:freeenergy}
\end{equation}
where $\chi>0$ is the interaction coefficient, and $m\in (1,2)$ is the non linear diffusion exponent. Here, $\rho$ is a probability distribution function such that $\rho\in L^m\cap L^1(|x|^2dx)$.

Such functionals arise as free energy functionals for models of self-interacting, diffusing particles \cite{LiebLoss97,McCann97}, for instance the Keller-Segel model and its non-linear variants (see \cite{BCL09,CPS07,HiPa09,Breview13} and references therein). The approach we follow in this paper is currently restricted to the one-dimensional case. 

The condition $m>1$ expresses that the functional has negative homogeneity in a particular sense (see below), whereas the condition $m<2$ guarantees integrability of the interaction kernel. The latter condition will disappear as soon as we switch to a discrete problem.

Homogeneity of \eqref{eq:freeenergy} under mass-preserving dilations reads as follows,
\[  \mathcal G_m[\rho_\lambda]  = \lambda^{1-m} \mathcal G_m[\rho ]\, , \quad \text{where $\rho_\lambda(x) = \lambda^{-1}\rho(\lambda^{-1} x)$}\, . \]

After the usual Lagrangian transformation $\rho\mapsto X$, where $X:(0,1)\to \R$ denotes the pseudo-inverse of the cumulative distribution function associated with  $\rho$ \cite{topicsinOT,GosTos06,CC12}, the functional \eqref{eq:freeenergy} reads equivalently, 
\begin{equation} \label{eq:energy lagrangian} 
\F_{m}(X) =  \frac{1}{m-1}\int_{(0,1)}  \left(\dfrac {dX}{dp}(p)\right)^{1-m}\, dp - \frac{\chi}{m-1} \iint_{(0,1)^2} |X(p)- X(q)|^{1-m}\, dpdq\, .
\end{equation}
Intuitively, $X$ encodes the position of particles with respect to the partial mass $p\in (0,1)$. Homogeneity of \eqref{eq:energy lagrangian} reads as usual, $\F_{m}(\lambda X) = \lambda^{1-m}\F_{m}(X)$.

Consider the gradient flow of $\mathcal G_m$, resp. $\F_{m}$ in the Wasserstein metric space, resp. the $L^2$ Hilbert space (see \cite{Otto01, oldandnew, AGSbook} for precise concepts and definitions), which is written in the abstract way as follows,
\[ \dot X(t) = - \nabla \F_m(X(t))\, . \]
Then, it is an immediate consequence of homogeneity that 
\begin{equation}  \label{eq:1st derivative X^2} \dfrac d{dt}\left( \frac12 |X(t)|^2\right) = X(t)\cdot \dot X(t) = - X(t)\cdot \nabla \F_m(X(t)) = (m-1)\F_m(X(t))\, . \end{equation}
Moreover, we deduce from the gradient flow structure that,
\begin{equation} \label{eq:2nd derivative X^2} \dfrac {d^2}{dt^2}\left( \frac12 |X(t)|^2\right) = (m-1) \dfrac d{dt} \F_m(X(t)) = -(m-1)|\nabla \F_m(X(t))|^2\, . \end{equation}
Therefore, $|X|^2$ is concave along any trajectory. If it is decreasing at some time, then it must vanishes in finite time. The latter condition is equivalent to the existence of $t_0$ such that $\F_m(X(t_0))<0$. We refer to \cite{BCL09} for a similar result in the Keller-Segel problem with non linear diffusion.

The following question arises naturally: does there exist $t_0$ such that $\F_m(X(t_0))<0$ for any initial data? Alternatively speaking, do solutions always blow-up? A first condition is clearly that the energy functional $\F_m$ possesses a cone of negative energy. In the class of functionals \eqref{eq:energy lagrangian}, this requires $\chi$ to be large enough, namely $\chi>\chi_m$, for some threshold value $\chi_m$ \cite{BCL09}.

The purpose of this work is to answer positively this question, but for a finite dimensional approximation of \eqref{eq:energy lagrangian}, which complies with the same algebra as $\F_m$. More precisely, it is required that similar identities as \eqref{eq:1st derivative X^2} and \eqref{eq:2nd derivative X^2} hold for the approximate system.

Let $m>1$. The discrete energy functional we study here is deduced from a finite difference approximation of $X(p)$ on a regular grid. Let $(X_i)_{1\leq i\leq N}$ be the positions of $N$ ordered particles sharing equal mass $1/N$, such that 
$X_1<X_2<\dots<X_N$.  
The discrete functional is defined by analogy with
\eqref{eq:energy lagrangian}:
\begin{equation}\label{EnergieKelSelDiscret} 
\F^N_{m}\left(X\right)=\frac{1}{m-1}\sum_{i=1}^{N-1} \left(X_{i+1}-X_i\right)^{1-m} -\frac{\chi}{m-1} \sum_{1\le i\neq j\le N}  \vert X_j-X_i \vert^{1-m} \,,
\end{equation}
where we have renormalized $\chi$ in order to absorb prefactors involving the fraction of mass $1/N$.

\begin{Rk}
For the sake of comparison, we recall that the limiting case $m\to 1$ was studied in \cite{BCC08,KS1Dp}. In this case, the functional reads as follows,
\begin{equation}\label{eq:log energy} 
\F^N_{1}\left(X\right)=-\sum_{i=1}^{N-1} \log\left(X_{i+1}-X_i\right) + \chi \sum_{1\le i\neq j\le N}  \log \vert X_j-X_i \vert \,,
\end{equation}
We refer to the logarithmic case, or $m = 1$.
\end{Rk}
 
As the functional is translation invariant, we can assume w.l.o.g. that the particles are centered at 0. Accordingly, we define
\[
\Rord^N=\left\{ (X_i)_{1\leq i\leq N} \in \R^N \left| X_1<X_2<\dots<X_N \mbox{ and } \sum_{i=1}^N X_i=0  \right.  \right\}\, . 
\]
We denote by $|\cdot|$ the euclidean norm on $\R^N$.

Notice that $\F^N_{m}$ has the same homogeneity as $\F_N$. In addition, the threshold $\chi_m(N)$ is defined as
\[\dfrac1{\chi_m(N)} = \max_{X \in \Rord^N} \dfrac{\sum_{1\le i\neq j\le N}  \vert X_j-X_i \vert^{1-m}}{\sum_{i=1}^{N-1} \left(X_{i+1}-X_i\right)^{1-m}}\, .  \]
We can deduce easily from Jensen's inequality that the supremum is finite, therefore $\chi_m(N) >0$.

The euclidean gradient flow of $\F^N_{m}$ writes
\begin{equation}\label{flotgradientdiscretexplicite}
\dot X_i = -{\left(X_{i+1}-X_i\right)^{-m}} +  {\left(X_{i}-X_{i-1}\right)^{-m}} + \quad 2 \chi  \sum_{j \neq i } \sign(j-i){\left|X_j-X_i\right|^{-m}}\,,
\end{equation}
complemented with the dynamics of the extremal points
\begin{equation}\label{flotgradientdiscretexplicite-boundary}
\left\{\begin{array}{ccl}\dot X_1 =& -{\left(X_{2}-X_1\right)^{-m}} + \quad 2 \chi  \sum_{j \neq 1 } {\left|X_j-X_1\right|^{-m}}  \medskip\\
\dot X_N=&    {\left(X_{N}-X_{N-1}\right)^{-m}} - \quad 2 \chi  \sum_{j \neq N }{\left|X_j-X_N\right|^{-m}}
\end{array}\right.
\end{equation}

Before we state our main result, we introduce  a decreasing family of threshold values $(C_p)$ parametrized by the integers $p = 1\dots N$, for a given $m>1$.
\begin{equation}
\dfrac1{C_p} = \max_{X \in \Rord^N} \dfrac{\sum_{1\le i\neq j\le p}  \vert X_j-X_i \vert^{1-m}}{\sum_{i=1}^{p-1} \left(X_{i+1}-X_i\right)^{1-m}}\, .  
\end{equation}
Accordingly, we have $C_N = \chi_m(N)$. 
\begin{Def}
For the discrete problem \eqref{flotgradientdiscretexplicite}--\eqref{flotgradientdiscretexplicite-boundary} we refer to the subcritical regime when $\chi < C_N$, and to the super critical regime when $\chi >C_N$. The critical case refers to $\chi = C_N$.
\end{Def}
The next couple of Theorems give an almost complete description of possible dynamics of \eqref{flotgradientdiscretexplicite}--\eqref{flotgradientdiscretexplicite-boundary}.

Firstly, global existence holds in the subcritical regime.
\begin{Thm}[Subcritical regime]\label{thm:subcase}
Let $m>1$, and $\chi<C_N$. Let $X_0\in \Rord^N$, and let $X$ be the solution of the system \eqref{flotgradientdiscretexplicite}--\eqref{flotgradientdiscretexplicite-boundary} with initial data $X(0) = X_0$. Then, the maximal time of existence of $X$ is $T = +\infty$.
\end{Thm}

\begin{Thm}[Super critical regime]\label{thm:explosionintro}
Let $m>1$, and $\chi>C_N$. Let $X_0\in \Rord^N$, and let $X$ be the solution of the system \eqref{flotgradientdiscretexplicite}--\eqref{flotgradientdiscretexplicite-boundary} with initial data $X(0) = X_0$. Then, up to a finite range of values for $\chi$, the maximal time of existence is finite, $T<+\infty$, without any condition on $X_0$. More precisely,
\begin{enumerate}
\item If $\chi\notin (C_p)_{p = 1\dots N-1}$, then $X$ blows up in finite time.
\item If there exists an integer $p \in [1,N-1]$ such that $\chi = C_p$, then either $X$ blows up in finite time or the renormalized location $Y=\frac{X}{|X|}$ blows up in infinite time. Moreover, any relative blow-up set of $Y$ obtained after extraction of a subsequence $Y(t_n)$, contains exactly $p$ particles. In the latter case, the restriction of  $Y(t_n)$ to the blow-up set converges after subsequent renormalization to the unique critical point of $\F^p_m$. 
\end{enumerate}
\end{Thm}

\begin{figure}
\includegraphics[width=0.52\linewidth]{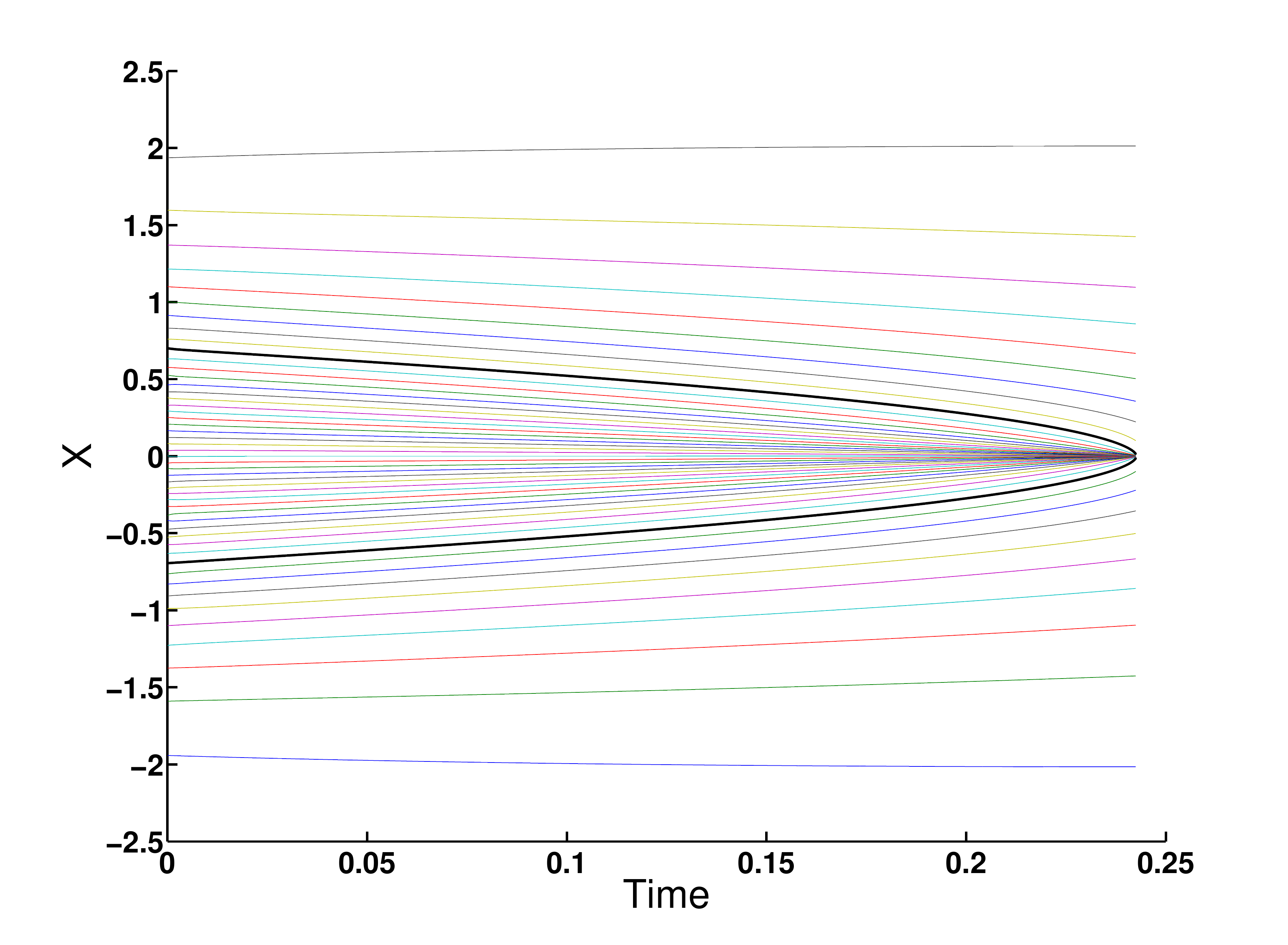}\quad 
\includegraphics[width=0.45\linewidth]{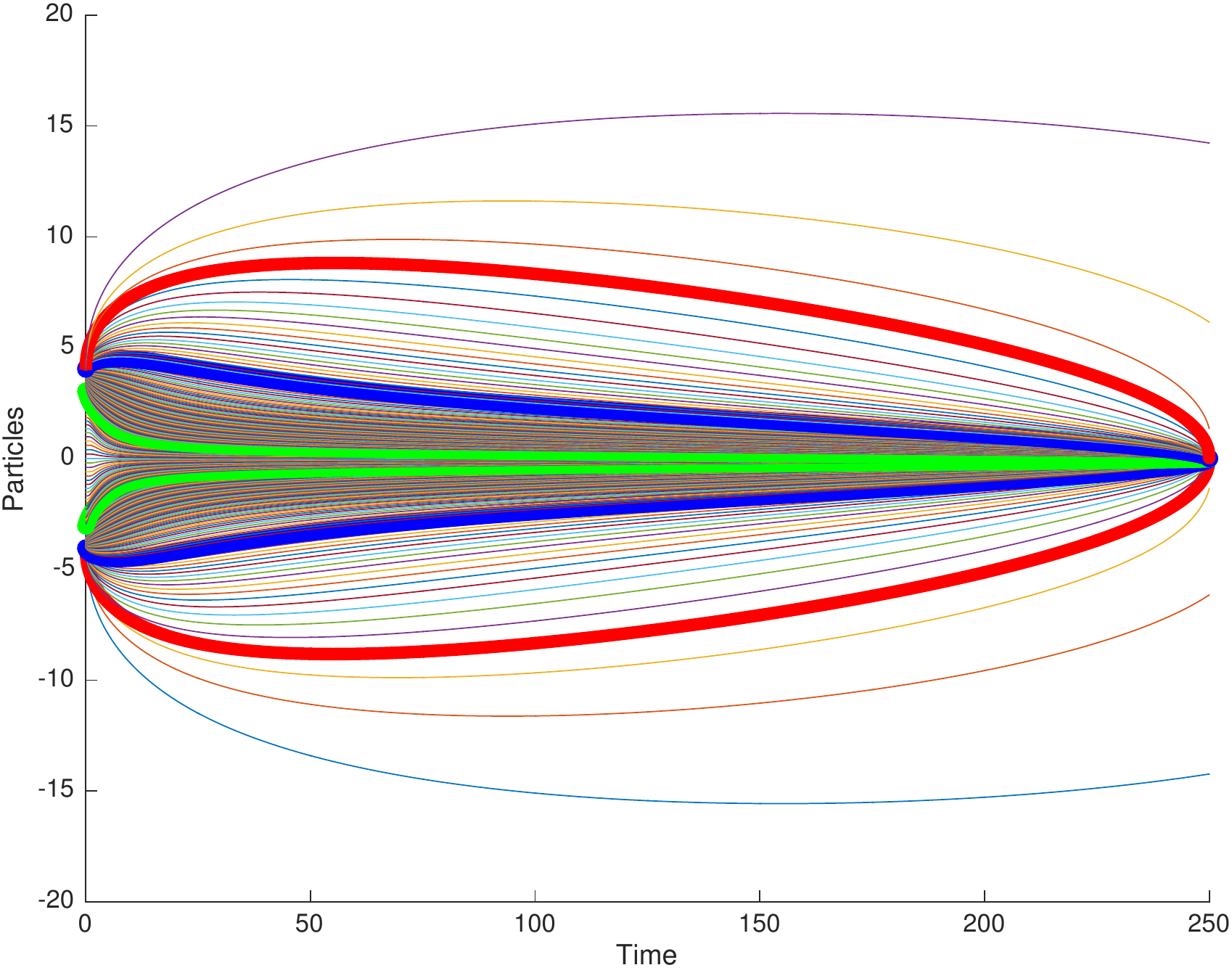} \\
\includegraphics[width=0.52\linewidth]{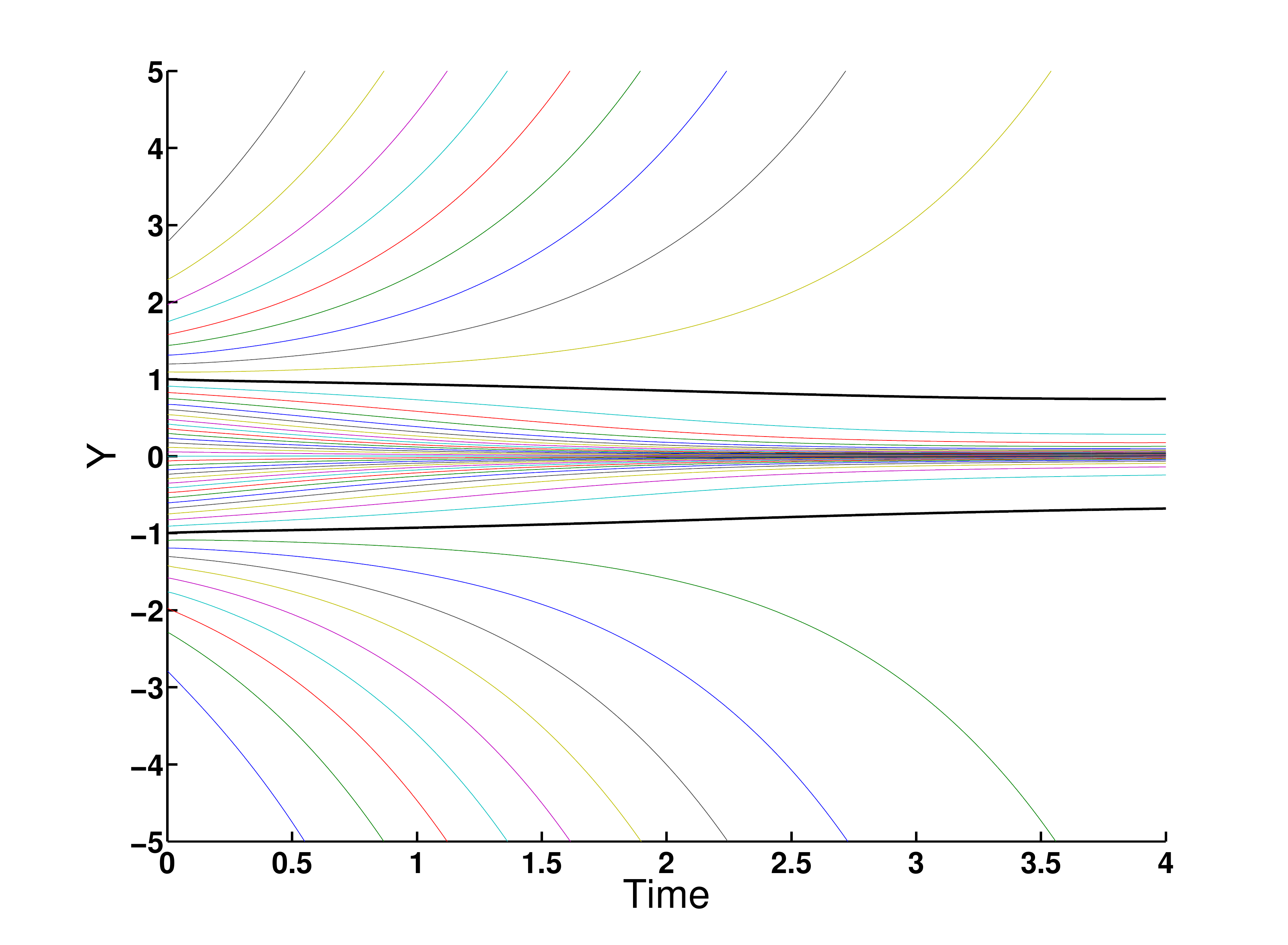}\quad 
\includegraphics[width=0.45\linewidth]{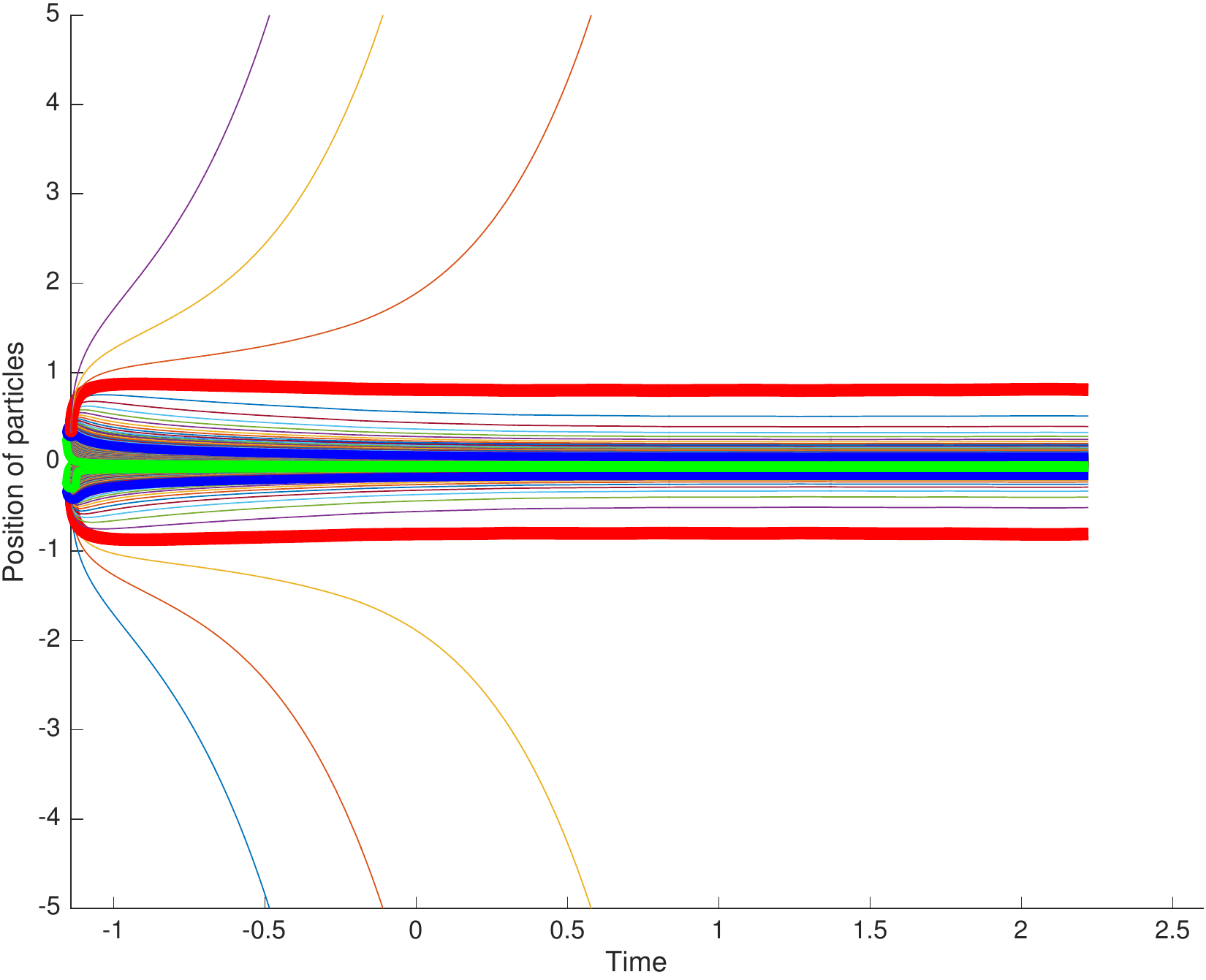}\\
\caption{Comparison of typical blow-up dynamics for a discrete gradient flow with 50 particles. Left panel corresponds to the case of logarithmic homogeneity \eqref{eq:log energy}, and Right panel corresponds to homogeneity $1-m$, with $m>1$ \eqref{EnergieKelSelDiscret} ($m = 1.2$ here). Top figures are in original scales. We see that a subset of particles concentrate at the origin in finite time. Dynamics seem quite different in both cases. Bottom figures are obtained after a proper parabolic rescaling (see \cite{KS1Dp} for details about this procedure. Particles outside the blow-up set are sent to infinity, whereas particles inside the blow-up set are distributed asymptotically along some profile which is presumably the profile of some renormalized functional energy. We do not address this last issue in the present work, but we focus on the unconditional blow-up in the super-critical regime. See Section \ref{sec:Numerics} for more details about the numerical procedure.}
\label{explosionkpamisnreg}
\end{figure}

The definition of a relative blow up set is given in Definition \ref{blowup}. Loosely speaking, it consists in the set of particles participating in the core of the blow-up, up to extraction.
Our last result deals with the critical case $\chi=C_N$.
\begin{Thm}[Critical case]\label{thm:explosionintro3}
Let $m>1$, and $\chi>C_N$. Let $X_0\in \Rord^N$, and let $X$ be the solution of the system \eqref{flotgradientdiscretexplicite}--\eqref{flotgradientdiscretexplicite-boundary} with initial data $X(0) = X_0$. Then, the maximal time of existence of $X$ is $T = +\infty$. Moreover, if $|X(t)|$ is uniformly bounded, then $Y=\frac{X}{|X|} $ converges to the unique critical point of $\F_m$ with unit norm.
\end{Thm}

This is somehow unsatisfactory to distinguish between the case $\chi = C_p$ for some $p$, and $\chi\neq C_p$. However, we believe that there are strong differences in the dynamics that rule out a unified result for all $\chi>C_N$. We already encountered such restrictions in the refined description of blow-up sets in the logarithmic case $m = 1$ \cite{KS1Dp}.

Our approach cannot readily be extended to the infinite-dimensional case of \eqref{eq:energy lagrangian}. Indeed, we crucially use discrete functional inequalities that we are not able to establish in the continuous setting.

We certainly miss some complicated dynamics specific to the continuous setting. On the other hand, our analysis does not rely on any perturbation argument but mostly rely on homogeneity. Finally, we emphasize that we are able to recover a very nice dichotomy, analogous to the two-dimensional Keller-Segel model, or the one-dimensional Keller-Segel equation with a logarithmic interaction kernel \cite{CPS07,BCC08}, for which the present finite-dimensional reduction yields \eqref{eq:log energy}. 

However, the case $m=1$ is much simpler, because the logarithmic homogeneity $\F_{1}(\lambda X) = \F_{1}(X) + (N-1)(-1+\chi N/2)$ immediately implies the following relation,
\begin{equation*}  \dfrac d{dt}\left( \frac12 |X(t)|^2\right) = (N-1)(1-\chi N/2)\, . \end{equation*}
Therefore, blow-up necessarily occurs in finite time when $\chi>2/N$. Alternative arguments show that global existence holds when $\chi<2/N$. It is remarkable that Theorems \ref{thm:subcase} and  \ref{thm:explosionintro} can reproduce such a dichotomy although the dynamics are much more nonlinear (the second moment is not monotonic, for instance). For the sake of comparison, typical blow-up dynamics are plotted in Figure \ref{explosionkpamisnreg}.

To conclude this introduction, we put our results in a more general context. As mentioned above, similar results were obtained for the high-dimensional Keller-Segel system ($d\geq 3$) with nonlinear diffusion by Sugiyama \cite{Sugiyama} and  Blanchet, Carrillo and Lauren\c{c}ot \cite{BCL09},
\begin{equation}\label{eq:BCL}
\begin{cases}
\partial_t \rho  = \Delta \rho^{m}-  \chi \nabla \cdot \left(\rho\nabla S\right) &\, , \quad t>0,\;x\in  \R^d \, ,  \medskip\\
- \Delta S = \rho 
\end{cases} 
\end{equation}
in the special case $m = 2-2/d$. The free energy functional associated to this equation is 
\[
\mathcal G[\rho]  = \frac{d}{d-2} \int_{\R^d} \rho (x)^{2-2/d} \, dx -\frac{\chi }{d-2} \iint_{\R^d\times \R^d} |x-y|^{2-d}\rho(x)\rho(y)\, dx dy\, . 
\]
With this specific choice of exponents, both contributions in the functional have the same homogeneity under mass-preserving dilations, $\rho_\lambda(x) = \lambda^{-d}\rho(\lambda^{-1} x)$.
The Wasserstein metric confers a gradient flow structure to the system \cite{Otto01,AGSbook}. 
The authors proved that initial data with negative energy yields blow-up in finite time \cite[Lemma 4.1]{BCL09}. However, they leave the case of positive energy open. 

More recently, on the same problem \eqref{eq:BCL}, Yao used comparison principles valid under radially symmetric conditions, in order to prove blow-up for a large class of initial data, including the ones having positive energy \cite{Yao13}. Note that she crucially used the fact that the interaction kernel is the Green's function of the $d-$dimensional laplacian, in order to reduce the problem to a local equation on cumulative mass inside balls.


The paper is structured as follows: 
in Section \ref{sec:critiquepourtous} we discuss links between homogeneity, lower bounds, critical points and self similar dynamics for the homogeneous functionals given by $\F^N_m$. This section is very general and only uses homogeneity. 
In Section \ref{sec:critiquedoncmini} we investigate the links between critical points and minimizers for the functionals $\F^N_m$. This section uses homogeneity and optimal transports arguments. 
Section \ref{sec:BU} is devoted to the proof of Theorem \ref{thm:explosionintro} and is specific to the particle system, in particular throughout compactness arguments. In Section \ref{sec:conc} we investigate the (sub)-critical case, prove Theorem \ref{thm:explosionintro3}, Theorem \ref{thm:subcase} and discuss the perspectives of our work. 
Finally we complete our analysis with numerical results in Section \ref{sec:Numerics}. 

\paragraph{Acknowledgement.} This project has received funding from the European Research Council (ERC) under the European Unions Horizon 2020 research and innovation programme (grant agreement No 639638). T. O. Gallou\"et was supported by the ANR contract ISOTACE (ANR-12-MONU-013).

\section{Critical points of homogeneous functionals}\label{sec:critiquepourtous}

Firstly, we define the energy functionals under consideration, including quadratic potential energy. 


\begin{Def}\label{def:fonctionnelle} Let $\chi>0$, $\alpha\in \R$ and $ m > 1 $ 
We define $\F^N_{m,\alpha} $ on $\Rord^N$ by
\begin{align*}
 \F^{N}_{m,\alpha}(X) &= \dfrac1{m-1}\sum_{i=1}^{N-1} \left(X_{i+1}-X_i\right)^{1-m} - \dfrac\chi{m-1} \sum_{1\le i\neq j\le N}  \vert X_i-X_j \vert^{1-m}+\alpha\dfrac{|X|^2}2 \\
 &= \U_{m}(X) - \chi \W_{1-m}(X) +\alpha \V (X)
 \end{align*}
In particular, we denote $\F^N_{m}=\F^N_{m,0}$.
\end{Def}
Functionals $\U_{m}, \W_{1-m}$, and $\V$ are resp. the internal energy, the interaction potential and  the quadratic  potential. Notice that we do not impose any sign condition on $\alpha$.

The homogeneity of the functionals $ \F^N_{m} $ plays a crucial role. We recall some useful formulas is the following proposition.

\begin{Prop}\label{LemHomo}
For all $X\in \Rord^N$ and $\lambda >0$, we have
\begin{enumerate}
\item $\F^N_m\left( \lambda X\right)= \lambda^{1-m} \F^N_m\left(  X\right),$
\item  $\nabla \F^N_m\left( \lambda X\right)= \lambda^{-m} \nabla \F^N_m\left(  X\right),$
\item $X \cdot \nabla \F^N_m\left( X\right)=-(m-1) \F^N_m\left(  X\right).$
\end{enumerate}
\end{Prop}

We define a sequence of treshold values $(C_p)$ for $p\in \N^*$, each being the optimal constant of a discrete Hardy-Littlewood-Sobolev inequality (see \cite{BCL09,Sugiyama} for a continuous version of it). 

\begin{Prop}\label{thm:bornepardessous}
Let $p\in \N^*$. We define $C_p$ as
\[\dfrac1{C_p} =\max_{X \in \Rord^p} \dfrac{\W_{1-m}}{\U_{m}} =  
\max_{X \in \Rord^p} \dfrac{\sum_{1\le i\neq j\le p}  \vert X_j-X_i \vert^{1-m}}{\sum_{i=1}^{p-1} \left(X_{i+1}-X_i\right)^{1-m}} =  \leq p \,. \]
Moreover, if $\alpha \geq 0$, then the functional $\F^{p}_{m,\alpha}$ is bounded  below iff $ \chi \leq C_p$.
 \end{Prop}

\begin{proof}
First, we observe that $C_p\geq \frac{1}{p}$. This is a consequence of the trivial inequality $X_j-X_i \geq X_{i+1} - X_i$ for $j>i$, combined with $1-m<0$,
\begin{align*} 
\sum_{1\le i\neq j\le p}  \vert X_i-X_j \vert^{1-m} &\leq \sum^{p-1}_{i =1}  \sum^p_{ j= i+1} \vert X_{i+1}-X_i \vert^{1-m} + \sum^{p}_{i =2}  \sum^{i-1}_{ j= 1} \vert X_{i}-X_{i-1} \vert^{1-m}\\
&\leq \sum^{p-1}_{i =1} (p-i)  \vert X_{i+1}-X_i \vert^{1-m} + \sum^{p}_{i =2} (i-1) \vert X_{i}-X_{i-1} \vert^{1-m} \\ 
 &= p \sum^{p-1}_{i =1} \vert X_{i+1}-X_i \vert^{1-m}.
\end{align*}
Moreover, the maximum is reached for some critical $X\in \Rord^p$ since the interaction functional is continuous on the set $\{ \sum^{p-1}_{i =1} \vert X_{i+1}-X_i \vert^{1-m} = 1 \}$.

Assume $ \chi \leq C_p$ and $X\in \Rord^p$. By definition of $C_p$, we have,
\begin{equation*}
\F^{p}_{m,\alpha}(X)  \geq  ( C_p - \chi ) \sum_{1\le i\neq j\le p}  \vert X_i-X_j \vert^{1-m}+\frac{\alpha}{2} | X|^2 \geq 0\, . 
\end{equation*}
On the contrary, assume $\chi > C_p$. By definition of $C_p$, there exists $X\in \Rord^p$ such that $\F^p_m(X) < 0$. By homogeneity of the functionals, we get $\F^p_{m,\alpha}(\lambda X) = \lambda^{1-m}\F^p_m(X) + \lambda^2 \displaystyle \frac{\alpha}{2} |X|^2 $, which is not bounded below as $\lambda$ goes to $0$.
\end{proof}

Considering the dilations or equivalently some self similar dynamics for the gradient flow system \eqref{flotgradientdiscretexplicite}--\eqref{flotgradientdiscretexplicite-boundary}, we can state a first proposition on the existence of critical points for the functionals $\F^p_{m,\alpha} $. 

The next proposition states some properties of the critical points of $\F^p_{m,\alpha}$, depending on $\chi$, and the sign of $\alpha$. 

\begin{Prop}\label{cro:nocritiquetouscas}
Let $\chi>0$ and $p\in \N^*$. 
\begin{enumerate}[(i)]
\item If $\chi < C_p$ and $\alpha\leq 0$, the functional $\F^p_{m,\alpha}$ has no critical point.
\item If $\chi = C_p$ and $\alpha < 0$, the functional $\F^p_{m,\alpha}$ has no critical point.
\item If $\chi > C_p$, any  critical point $V$ of $\F^p_{m,\alpha}$ satisfies the identity
\begin{equation}\label{valeurmin}
\F^p_{m}\left(V\right)=\frac{2}{m-1} \alpha\V(V)\, .
\end{equation}
In particular, $\F^p_{m}\left(V\right)$ has the same sign as $\alpha$.
\end{enumerate}
\end{Prop}
\begin{proof}
Let $V$ be a critical point of $\F^p_{m,\alpha}$. The third identity in Proposition \ref{LemHomo} implies 
\begin{align*} 
0 & = \nabla\F^p_{m}(V) + \alpha V  \\
0 & = V\cdot\nabla\F^p_{m}(V) + \alpha |V|^2\\ 
0 & = -(m-1) \F^p_{m}(V) + \alpha |V|^2  \, .   \end{align*}
The latter is equivalent to \eqref{valeurmin}. This proves the third point (iii). The two other facts are consequences of \eqref{valeurmin}. 

\noindent(i) We deduce from the very definition of $C_p$ that, if $\chi< C_p$, then  $\F^p_{m}\left(X\right)> 0$ for all $X\in \Rord^p$. Necessarily, we have $\alpha>0$ in this case. 

\noindent(ii) Similarly, if $\chi =  C_p$, then $\F^p_{m}\left(X\right)\geq  0$ for all $X\in \Rord^p$. Necessarily, we have $\alpha\geq 0$ in this case. 
\end{proof}

\begin{Rk}
Another way to prove Proposition \ref{cro:nocritiquetouscas}  is to consider the solution $X(t)$ of the gradient flow  \eqref{flotgradientdiscretexplicite}--\eqref{flotgradientdiscretexplicite-boundary} with initial data $V$, critical point of $\F^p_{m, \alpha}$. The solution is self-similar, and satisfies $X(t)= \lambda(t) V$, where $ \lambda(t)=\left(1+ \frac{\alpha}{m+1} t\right)^{1/(m+1)}.$ 
\end{Rk}
\begin{Rk}
We have similar results in the limit case $m=1$.
\end{Rk}

In the next section we investigate further the existence of critical points for the functionals $\F^p_{m,\alpha}$. 

\section{Critical points are minimizers}\label{sec:critiquedoncmini}

\begin{Thm}\label{thm:cridonnemin}
Let $\chi>0$, and $p\in \N^*$. If $\alpha\geq 0$, then any critical point of $\F^p_{m,\alpha}$ is a minimizer of $\F^p_{m,\alpha}$.\\
In addition, when $\chi = C_p$, there exists a unique minimizer of $\F^p_{m}$ in $\Rord^p$, up to dilation.
\end{Thm}

\begin{Rk}
The proof of Theorem \ref{thm:cridonnemin} is inspired by optimal transport techniques: we transport any configuration to some critical point, then we use a convexity inequality to conclude. Similar ideas can be found in \cite{BCC08}[Proposition 4.4] and \cite{Franca}.
\end{Rk}

\begin{proof}[Proof of Theorem \ref{thm:cridonnemin}]
First, we notice that, for $\chi = C_p$, minimizers of $\F^p_{m}$ are exactly maximal functions in Proposition \ref{thm:bornepardessous}. This proves existence of minimizers in this case. 

In order to prove that critical points are minimizers, we begin with a characterization of  critical points. 
Recall that, by convention, we set $\left|X_{p+1}-X_{p}\right |^{1-m}=\left|X_{1}-X_{0}\right |^{1-m}=0$. Let $V\in \Rord^p$ be a critical point of $\F^p_{m,\alpha}$. For all $k\in [1,p]$, we have:
\begin{equation}\label{eq:criticalpoint}
-\left(V_{k+1}-V_k\right)^{-m} + \left(V_{k}-V_{k-1}\right )^{-m} + 2 \chi \sum_{j \neq k} \sign(j-k) \left|V_{j}-V_k\right |^{-m} - \alpha V_k=0.
\end{equation}
After summation, and using conservation of the center of mass, $\sum_{i=1}^p V_i=0$, the latter relations are equivalent to the following,
\begin{equation}\label{eq:criticalpoint2}
\left(V_{k+1}-V_k\right)^{-m} = 2 \chi \sum^k_{i =1} \sum^p_{j =k+1} \left(V_{j}-V_i\right)^{-m} + \frac\alpha{p}\sum^k_{i =1} \sum^p_{j =k+1}  \left(V_j-V_i\right)  \,.
\end{equation}
Relation \eqref{eq:criticalpoint} can be deduced   from \eqref{eq:criticalpoint2}, using discrete derivation, 
\begin{align*}
\left(V_{k+1}-V_k\right)^{-m}-\left(V_{k}-V_{k-1}\right)^{-m} &= 2 \chi \sum^k_{i =1} \sum^p_{j =k+1} \left(V_{j}-V_i\right)^{-m}+\frac\alpha{p}\sum^k_{i =1} \sum^p_{j =k+1}  \left(V_j-V_i\right)\\
&\quad -2 \chi \sum^{k-1}_{i =1} \sum^p_{j =k} \left(V_{j}-V_i \right)^{-m}-\frac\alpha{p}\sum^{k-1}_{i =1} \sum^p_{j =k}  \left(V_j-V_i\right)\\
&= 2 \chi \left[  \sum^p_{j =k+1} \left(V_{j}-V_k\right)^{-m}- \sum^{k-1}_{i =1}  \left(V_{k}-V_i\right)^{-m}  \right] \\
& \quad - \frac\alpha{p}\left[- \sum^p_{j =k+1}  \left(V_j-V_k\right) +\sum^{k-1}_{i =1}   \left(V_k-V_i\right)\right]\\
&= 2 \chi  \sum_{j \neq k} \sign(j-k) \left|V_{j}-V_k\right|^{-m} -\frac\alpha{p}\left[ p V_k + \sum^p_{j =1}  V_j \right]\\
&= 2 \chi  \sum_{j \neq k} \sign(j-k) \left|V_{j}-V_k\right|^{-m} -\alpha V_k.\\
\end{align*}

We divide the rest of the proof in two cases, resp. $\alpha = 0$, and $\alpha>0$. 

\noindent{\bf Case 1: $\alpha=0$.} For sake of clarity, we begin with the case $\alpha =0$ which simply reduces to Jensen's inequality.  The strategy is to make appear, for each contribution in $\F^p_{m}$, the relative energy from $X$ to $V$. For the interaction potential, we have:
\begin{align*}
(m-1)\W_{1-m}(X)&=2\chi \sum_{1\le i< j\le p}  ( X_j-X_i )^{1-m} =2\chi \sum_{1\le i< j\le p} \left(\frac{X_{j}-X_i}{V_{j}-V_i}\right)^{1-m} \left(V_{j}-V_i\right)^{1-m}\\
&=2\chi \sum_{1\le i< j\le p} \left(\sum^{j-1}_{ k=i} \frac{X_{k+1}-X_k}{V_{k+1}-V_k} \frac{V_{k+1}-V_k}{V_{j}-V_i}\right)^{1-m} \left(V_{j}-V_i\right)^{1-m}.\\
\end{align*}
Since, for any $(i,j)$, $\sum^{j-1}_{ k=i} \left(\frac{V_{k+1}-V_k}{V_{j}-V_i}\right)=1$, a successive use of  Jensen's inequality implies,
\begin{align} \nonumber
2\chi \sum_{1\le i\neq j\le p}  \vert X_i-X_j \vert^{1-m} 
&\leq 2\chi \sum_{1\le i< j\le p}\sum^{j-1}_{ k=i} \left(\frac{X_{k+1}-X_k}{V_{k+1}-V_k} \right)^{1-m} \left(\frac{V_{k+1}-V_k}{V_{j}-V_i}\right) \left(V_{j}-V_i\right)^{1-m}\\ \nonumber
&= 2\chi \sum^{p-1}_{ k =1} \sum^k_{i=1} \sum^p_{j=k+1} \left(\frac{X_{k+1}-X_k}{V_{k+1}-V_k} \right)^{1-m} \left(\frac{V_{k+1}-V_k}{V_{j}-V_i}\right) \left(V_{j}-V_i\right)^{1-m}\\ \nonumber
&=  \sum^{p-1}_{ k =1} \left({X_{k+1}-X_k} \right)^{1-m}\left[ \left(V_{k+1}-V_k \right)^{m}2\chi \sum^k_{i=1} \sum^p_{j=k+1}  \left(V_{j}-V_i\right)^{-m}\right]\\ \label{calcul}
&=  \sum^{p-1}_{ k =1} \left({X_{k+1}-X_k} \right)^{1-m}.
\end{align}
To get to the last line of the argument, we used the characterization of critical points obtained in  \eqref{eq:criticalpoint2}. 
We deduce that, for all $X\in \Rord^p$, $\F^p_{m}(X) \geq 0 $. Equality occurs in Jensen's inequalities  if and only if there exists $\lambda >0$ such that $X=\lambda V$. In particular, we recover that critical points have zero energy,  $\F^p_{m}(V)=0$, which can be deduced directly from homogeneity.

\begin{Rk}
The case $\alpha = 0$ is only compatible with $\chi = C_p$. Indeed, the functional $\F^p_{m}$ has no critical point for $\chi >C_p$. Otherwise, we would get that $\F^p_{m}(X) \geq 0 $ for all $X\in \Rord^p$, which is in contradiction with the definition of $C_p$.
\end{Rk}

\noindent{\bf Case 2: $\alpha>0$.} 
We begin with a variant of  Jensen's inequality. 
\begin{Lem}\label{newJensen}
Let $a,b>0$, $m>1$, $q\in \N^*$ and $w \in (\R_+^*)^q$. For any convex combination $\left(\lambda_k\right)_{k\in [1,q]}$, we have
\begin{equation}\label{newjensen}
\left(a+2b\right)  \sum^q_{k=1} \lambda_k \left(w_k \right)^{1-m} - a\left( \sum^q_{k=1} \lambda_k w_k \right)^{1-m} + b(m-1)\left( \sum^q_{k=1} \lambda_k w_k \right)^{2} \geq b(m+1).
\end{equation} 
Equality occurs if and only if $w\equiv 1$.
\end{Lem}
\begin{proof}
\begin{align*}
a &\left( \sum^q_{k=1} \lambda_k w_k \right)^{1-m} + b(1-m)\left( \sum^q_{k=1} \lambda_k w_k \right)^{2}  \\
& =  (a+2b) \left( \sum^q_{k=1} \lambda_k w_k \right)^{1-m}-b\left[2\left(   \sum^q_{k=1} \lambda_k w_k \right)^{1-m} + (m-1)\left( \sum^q_{k=1} \lambda_k w_k \right)^{2}\right] \\
&  \leq (a+2b) \sum^q_{k=1} \lambda_k  \left(  w_k \right)^{1-m} -b (m+1).
\end{align*}
We used Jensen's inequality on the first term. For the second term, we used the following scalar inequality,
\begin{equation}\label{adhoc}
(\forall x>0)\quad 2x^{1-m}+(m-1)x^2\geq m+1.
\end{equation}
Equality case for  Jensen's inequality requires that $w$ is a constant vector. The equality case in \eqref{adhoc} set the constant equal to one.
\end{proof}


We compute $\F^p_{m,\alpha}$, and again we make appear the relative energy from $X$ to $V$. 
Using the characterization of the critical point \eqref{eq:criticalpoint2}. The internal energy rewrites
\begin{align}  \nonumber
\U_{m}(X)&= \frac1{m-1}\sum_{k=1}^{p-1} \left(X_{k+1}-X_k\right)^{1-m}\\  \nonumber
& = \frac1{m-1}\sum_{k=1}^{p-1} \left( X_{k+1}-X_k\right)^{1-m} \left(V_{k+1}-V_k\right)^{m}  \sum^k_{i =1} \sum^p_{j =k+1}\left[ 2 \chi \left(V_{j}-V_i\right)^{-m}+\frac\alpha{p}  \left(V_j-V_i\right) \right] \\   \nonumber
& = \frac1{m-1}\sum_{k=1}^{p-1}  \sum^k_{i =1} \sum^p_{j =k+1} \left( \frac{X_{k+1}-X_k}{V_{k+1}-V_k} \right)^{1-m}  \left( \frac{V_{k+1}-V_k}{V_{j}-V_i} \right) \left[2 \chi \left(V_{j}-V_i\right)^{1-m}+ \frac\alpha{p} \left(V_j-V_i\right)^2\right]\\  \label{interne}  
& =  \sum_{1\le i< j\le p} \left( \sum^{j-1}_{ k=i}  \lambda_k^{i,j} \left( w_k \right)^{1-m} \right)\left[ \frac{2 \chi}{m-1} \left(V_{j}-V_i\right)^{1-m}+\frac\alpha{p(m-1)} \left(V_j-V_i\right)^2\right]\, , 
\end{align}
where we have denoted $w_k=\frac{X_{k+1}-X_k}{V_{k+1}-V_k}$, and $\lambda_k^{i,j}= \frac{V_{k+1}-V_k}{V_{j}-V_i}$.

The interaction can be reorganized as follows, 
\begin{align} \nonumber
\chi \W_{1-m}(X)&=\dfrac{2\chi}{m-1} \sum_{1\le i\neq j\le p}  \vert X_j-X_i \vert^{1-m} =2\chi \sum_{1\le i< j\le p} \left(\frac{X_{j}-X_i}{V_{j}-V_i}\right)^{1-m} \left(V_{j}-V_i\right)^{1-m}\\   \nonumber
&= \frac1{m-1}\sum_{1\le i< j\le p}2\chi \left(V_{j}-V_i\right)^{1-m} \left(\sum^{j-1}_{ k=i} \frac{X_{k+1}-X_k}{V_{k+1}-V_k} \frac{V_{k+1}-V_k}{V_{j}-V_i}\right)^{1-m}\\ \label{interaction}
&= \sum_{1\le i< j\le p}\frac{2\chi}{m-1} \left(V_{j}-V_i\right)^{1-m} \left(\sum^{j-1}_{ k=i} \lambda_k^{i,j} w_k\right)^{1-m} 
\end{align}

For the second moment we start with a doubling variable arguments, using $\sum^{p}_{i=1} X_i=0$.
\begin{align} \nonumber
\alpha\V(X)&=\frac\alpha2 \sum^{p}_{i=1} X_i^2 =\frac\alpha{4p} \sum^{p}_{i=1}\sum^{p}_{j=1} (X_j-X_i)^2  
= \frac\alpha{4p} \sum^{p}_{i=1}\sum^{p}_{j=1} \left( \frac{X_j-X_i}{V_j-V_i}\right)^2\left( {V_j-V_i}\right)^2  \\\nonumber
&= \frac\alpha{2p} \sum_{1\leq i<j\leq p}\left( \frac{X_j-X_i}{V_j-V_i}\right)^2\left( {V_j-V_i}\right)^2  \\ \nonumber
&=  \sum_{1\leq i<j\leq p} \frac\alpha{2p} \left( {V_j-V_i}\right)^2\left( \sum^{j-1}_{ k=i} \frac{X_{k+1}-X_k}{V_{k+1}-V_k} \frac{V_{k+1}-V_k}{V_{j}-V_i}\right)^2 \\  \label{secmoment}
&=  \sum_{1\leq i<j\leq p} \frac\alpha{2p} \left( {V_j-V_i}\right)^2\left( \sum^{j-1}_{ k=i} \lambda_k^{i,j} w_k\right)^2
\end{align}

Together, \eqref{interne}, \eqref{interaction}, \eqref{secmoment} and repeated use of \eqref{newjensen} with the weights $a^{i,j} = \frac{2 \chi}{m-1} \left(V_{j}-V_i\right)^{1-m}$ and $b^{i,j} = \frac\alpha{2p(m-1)} \left(V_j-V_i\right)^2$, imply the following inequality,
\begin{align*}
\F^p_{m,\alpha}(X) \geq \sum_{1\leq i<j\leq p}   \frac{\alpha}{2p} \frac{m+1}{m-1} \left( {V_j-V_i}\right)^2=    \frac{m+1}{m-1} \alpha \V(V) 
=\F^p_{m,\alpha}(V)\, .
\end{align*}
According to Lemma \ref{newJensen}, equality occurs if and only if  $X=V$. The equality case is more restrictive when $\alpha>0$ since homogeneity is  broken.
Notice that the equality $\F^p_{m,\alpha}(V)=\alpha \frac{m+1}{m-1} \V(V)$ can be deduced immediately from  \eqref{valeurmin}.
\end{proof}

The following Corollary gives a summary of Theorem \ref{thm:cridonnemin} that will be useful in the sequel. It also completes \ref{cro:nocritiquetouscas}.
\begin{Cor}\label{cro:nocritique}
Let $\chi>0$, $\alpha > 0$ and $p\in \N^*$. 
\begin{itemize}
\item If $\chi > C_p$, there exists no critical point for both $\F^p_m$ and $\F^p_{m,\alpha}$.
\item If $\chi < C_p$, there exists no critical point for $\F^p_m$.
\item If $\chi = C_p$ there exists no critical point for $\F^p_{m,\alpha}$.
\end{itemize}
\end{Cor}
\begin{Rk}
For the super critical case $\chi > C_p$ the improvement, as  compared to Proposition \ref{cro:nocritiquetouscas}, is to show that there is no critical point for $\F^p_m$ and no critical point $V$ for $\F^p_{m,\alpha}$ such that $\F^p_m(V)\geq 0$.
\end{Rk}
\begin{proof}
Let $\alpha\geq 0$. According to Theorem \ref{thm:cridonnemin}, a critical point of $\F^p_{m,\alpha}$ is a minimizer. Moreoveor, it has zero energy in the case $\alpha=0$. 

If $\chi > C_p$ then $\F^p_{m,\alpha}$ is not bounded below (see the proof of Proposition \ref{thm:bornepardessous}). Therefore, it cannot admit any critical point. 

If $\chi < C_p$, by definition of $C_p$, the functional $\F^p_m$ is positive for all $X\in \Rord^p$. Therefore, it cannot admit any critical point which would be a minimizer with zero energy. 

Finally, if $\chi = C_p$ and $\alpha>0$ then $\F^p_{m,\alpha}$ is positive for all $X\in {\Rord^p}$. Moreover there is a minimizing sequence with vanishing energy, namely $\lambda V$ where $V$ is a minimizer of $\F^p_m$ such that $\F^p_m(V) = 0$, and $\lambda \to 0$. Therefore, there cannot exist any global minimizer. 
\end{proof}

\section{Blow-up phenomena} \label{sec:BU}

We begin with some useful definitions and computations inspired by the case with logarithmic homogeneity \cite{KS1Dp}. 
Let $\In=[l,r] \subset [1,N]$ be some set of consecutive indices. We define
 $\Inl=[l,r-1]$, $\Inb = \left( \{l-1\} \cup \{ r \}\right) \cap [1,N-1]$, $\Out =[1,N] \setminus  \In$ and $\Outl =[1,N]\setminus  \Inl$. Observe that $\partial \Inl\subset \Outl$.  

\begin{Def}[Relative blow-up set]\label{blowup}
We say that $\In$ is a relative blow-up set for a sequence $(X^n)$ if 
\begin{equation}\label{strongblowupr}
\begin{array}{lll}\medskip
&(\forall i \in \Inl) \quad  &\underset{n \to +\infty }{\lim}\, \left(X^n_{i+1}-X^n_i \right)=0,\\
\medskip
&(\forall i \in \Inl)\; (\forall k \in [0,N-1]) \quad  &\underset{n \to +\infty }{\lim}\,   \displaystyle\frac{X^n_{i+1}-X^n_i}{X^n_{k+1}-X^n_{k}} < +\infty,\\
\medskip
&(\forall i \in \Inl)\; (\forall j \in \Inb)  \quad &\underset{n \to +\infty }{\lim}\, \displaystyle \frac{X^n_{i+1}-X^n_i}{X^n_{j+1}-X^n_j} =0.
\end{array}
\end{equation}
\end{Def}
The first condition states that particles in $\In$ are collapsing. The second condition states that nothing collapse faster than particles in $\In$. As a direct consequence, particles in $\In$ collapse at the same rate (take $k\in \Inl$). The third condition states that $\In$ is maximal for the rate of collapse. 

In summary, a relative blow-up set is a connected set having the fastest possible dynamics when blow up occurs. The notation $\Inl$ refers to the gaps between consecutive particles in $\In$, whereas $\Inb$ refers to the gaps between $\In$ and its nearest neighbours on each side. 

\begin{Def}
Let $T \in (0,+\infty]$ and $X \in C^0\left([0,T),\Rord^N\right)$ be a solution of the system \eqref{flotgradientdiscretexplicite}--\eqref{flotgradientdiscretexplicite-boundary}. We say that $X$ blows up if there exists some index $i_0$, and some sequence $t_n\to T$, such that $\lim \left(X_{i_0+1}(t_n)-X_{i_0}(t_n) \right)=0$. Then, it is always possible to build a relative blow-up set for $(X(t_n))$ after further extraction. 
\end{Def}

For a relative blow up set we can define a limiting profile.

\begin{Prop}\label{def:limitprofil}
Let $\In$ be a relative blow up set, associated with the sequence $(X^n)$. Let 
$p=|\In|$ and $\Pi_{\In}(X^n)=\sqrt{\underset{i \in \Inl}{\sum}\left( X^n_{i+1} -X^n_i \right)^2} $.  There exists a unique $Z\in \Rord^p$ such that
\begin{align}
    \begin{cases}
   (\forall i \in \Inl)\quad   Z_{i+1}-Z_i = \underset{n\rightarrow +\infty}{\lim}\displaystyle \frac{X^n_{i+1}-X^n_i}{\Pi_{\In}(X^n)} > 0 \,, & \\
    \sum^p_{i=1} Z_i =0\, . &\\
    \end{cases}
\end{align}  
\end{Prop}
\begin{proof}
For all $(i,k) \in  \Inl  \times \Inl$, we define 
\begin{align*}
\gamma^n_{i,k} =   \displaystyle\frac{X^n_{i+1}-X^n_i}{X^n_{k+1}-X^n_{k}}\,, \quad \mbox{and}\quad  \underset{n\rightarrow +\infty}{\lim}  \gamma^n_{i,k}=\gamma_{i,k}\, .
\end{align*}
By definition of a relative blow up set $0 < \gamma_{i,k}  < +\infty $. We deduce that
\begin{align*}
0<  \underset{n\rightarrow +\infty}{\lim} \displaystyle \frac{X^n_{i+1}-X^n_i}{\Pi_{\In}(X^n)}= \displaystyle \frac{1}{\sqrt{1+\sum_{ k \in \Inl\setminus\{i\}} \left(\gamma_{k,i}\right)^2}} < +\infty\, .
\end{align*}
Therefore, $Z$ is well defined up to an additive constant. We fix it, by imposing the center of mass to be zero. 
\end{proof}

The next proposition concerns the occurrence of blow-up in the super-critical regime, when the energy is negative initially. This is the easiest part, as explained in the Introduction.

\begin{Prop}[Blow-up with initial negative energy] \label{easycase}
Let $\chi>C_N$ and $X_0 \in \Rord^N$ such that $\F^N_m(X_0)<0$. Then, the solution $X(t)$ of the gradient flow system \eqref{flotgradientdiscretexplicite}--\eqref{flotgradientdiscretexplicite-boundary} blows up in finite time. 
\end{Prop}

\begin{proof}
Notice that we have \[\frac{d}{dt} \left( \F^N_m\left(X(t)\right) \right) = -\left| \nabla \F^N_m\left(X(t)\right) \right|^2\, . \] 
Therefore $(\forall t\geq 0)\; \F^N_m(X(t)) \leq \F^N_m(X_0) $.
We compute the second moment $f_2(t)=\frac12 |X(t)|^2$. It is a consequence of homogeneity (see Proposition \ref{LemHomo}) that 
\begin{align*}
\frac{d f_2}{dt}(t)= X(t) \cdot \nabla X(t)  = -X(t) \cdot \nabla \F^N_m(X(t)) = (m-1)\F^N_m(X(t)) \leq (m-1)\F^N_m(X_0)\, .
\end{align*}
Therefore, since it is assumed that $\F^N_m(X_0)<0$, the second moment decreases at least linearly, thus the maximal time of existence is necessarily finite. 
\end{proof}


To handle the case where the initial condition has positive energy, the most natural quantity to compute is some power function of the second moment. The next proposition contains the material of this computation.

\begin{Prop}
Let $\chi>0$ and $X_0 \in \Rord^N$. Let $X(t)$ be the solution of the gradient flow system \eqref{flotgradientdiscretexplicite}--\eqref{flotgradientdiscretexplicite-boundary}. We define 
$$
f_{m+1}(t)= \dfrac1{m+1} | X(t)|^{m+1}=    \dfrac1{m+1} \left(\sum^N_{i=1} X_i(t) ^2\right)^{(m+1)/2} .
$$
The function $f_{m+1}$ is concave. Moreover, we have the following identities:
\begin{align}
\label{dynmomentun}
\frac{d f_{m+1}}{dt}(t)  &=   (m-1) \F^N_m(X(t)) | X(t)|^{m-1}, \\
\label{dynmomentdeux}
\frac{d^2 f_{m+1}}{dt^2}(t)  &=-\frac{ m-1}{m+1}\left(f_{m+1}\left(t\right)\right)^{-1}  \H^N_{m+1}\left( \frac{X(t)}{|X(t)|}\right),
\end{align}
where for any $Y \in \Rord^N$ with $|Y|=1$, $\H^N_{m+1}$ is defined as 
\begin{equation}
 \H^N_{m+1}\left(Y \right) = \left[ \left| \nabla \F^N_m\left(   Y  \right) \right|^2  - \left((m-1)\F^N_m\left( Y  \right)   \right)^2 \right] \geq 0.
 \end{equation}
 \end{Prop}
 \begin{Rk}
In the preceding computation, $m+1$ is the maximal exponent $r$ for which the function  $f_{r}$ is concave. This quantity  already appeared  in \cite{DolTos13,DolTos15},  in the analysis of long-time behaviour for porous medium equations .
 \end{Rk}
\begin{proof}
We start by proving formula \eqref{dynmomentun} and \eqref{dynmomentdeux}, by using the homogeneity of the functional $\F^N_m$. 
\begin{align}
\nonumber \frac{d f_{m+1}}{dt}(t) &=  \left( X(t) \cdot \nabla X(t) \right) | X(t)|^{m-1} =   (m-1) \F^N_m(X(t)) | X(t)|^{m-1}. \\
\nonumber \frac{d^2 f_{m+1}}{dt^2}(t) &=  (m-1)\left[ -\left| \nabla \F^N_m\left(X(t)\right) \right|^2 | X(t)|^{m-1} + (m-1)^2\left(\F^N_m\left(X(t)\right)\right)^2 | X(t)|^{m-3}    \right] \\
\nonumber         &=  \frac{m-1}{|X(t)|^{m+1}} \left[ -\left| \nabla \F^N_m\left(X(t)\right) | X(t)|^{m} \right|^2  + (m-1)^2\left(\F^N_m\left(X(t)\right) | X(t)|^{m-1}  \right)^2 \right]\\ \label{deriveseconde1}
         &=  \frac{ m-1}{m+1}\left(f_{m+1}\left(t\right)\right)^{-1}  \left[ -\left| \nabla \F^N_m\left(   \frac{X(t)}{|X(t)|}  \right) \right|^2  + \left((m-1)\F^N_m\left( \frac{X(t)}{|X(t)|}\right)   \right)^2 \right]\\  \label{deriveseconde2}
         &= - \frac{ m-1}{m+1}\left(f_{m+1}\left(t\right)\right)^{-1}   \H^N_{m+1}\left( \frac{X(t)}{|X(t)|}\right)\, . 
\end{align} 
On the other hand, 
observe that
\begin{equation} \left((m-1)\F^N_m\left( Y  \right) \right)^2= \left( Y \cdot \nabla \F^N_m\left( Y\right) \right)^2 \leq \left( |Y| \left| \nabla \F^N_m\left(   Y  \right) \right| \right)^2\, , \label{eq:cauchy-schwarz H}
\end{equation}
by  Cauchy-Schwarz inequality. Therefore, $\H^N_{m+1}\geq 0$. 
\end{proof}

In the next lemma we investigate the functional $\H^N_{m+1}$ restricted to  the cone of nonnegative energy. The main result is that, apart from a finite of value for $\chi$, $\H^N_{m+1}$ is bounded below by a positive constant. 

\begin{Lem}\label{Lem:strictconcave}
Let $\chi \geq C_N$. We define
$$
\delta^N_{\H}=\inf \{ \H^N_{m+1}(Y)\left| Y \in \Rord^N,\quad  |Y|=1, \quad  \F^N_m(Y)\geq 0 \} \right.
$$
\begin{enumerate}
\item If for all integer $p\in [1,N]$, $\chi \neq C_p$ then $\delta^N_{\H}>0$. 
\item If $\chi=C_N$ then $\delta^N_{\H}=0$ and the unique minimizer is the critical point of $\F^N_m$ with unit norm. 
\item If there exists some $p \in [1,N-1]$ with $\chi=C_p$ then $\delta^N_{\H}=0$. There is no minimizer due to lack of compactness. There exists a relative blow up set up to extracting a subsequence. Moreover, all such relative blow up sets are made of 
$p$ consecutive particles converging towards a critical point of $\F^p_m$.
\end{enumerate}
\end{Lem}


\begin{proof}
Suppose that $\delta^N_{\H}=0$.  Let $\left(Y^n\right)$ be a minimizing sequence such that $\H^N_{m+1}(Y^n)\to 0$. By compactness, we can extract a subsequence converging to some $Y^\infty \in \overline{\Rord^N}$. We denote by $q$ the number of distinct values taken by the limiting profile $Y^\infty$. We denote by $(\Yinf_l)_{1\leq l \leq q}$ the ordered distinct values in the limit, $\Yinf_1 < \Yinf_2 < \dots < \Yinf_q$. See Figure \ref{fig:notations} for the notations.
\begin{figure}
\includegraphics[width=0.90\linewidth]{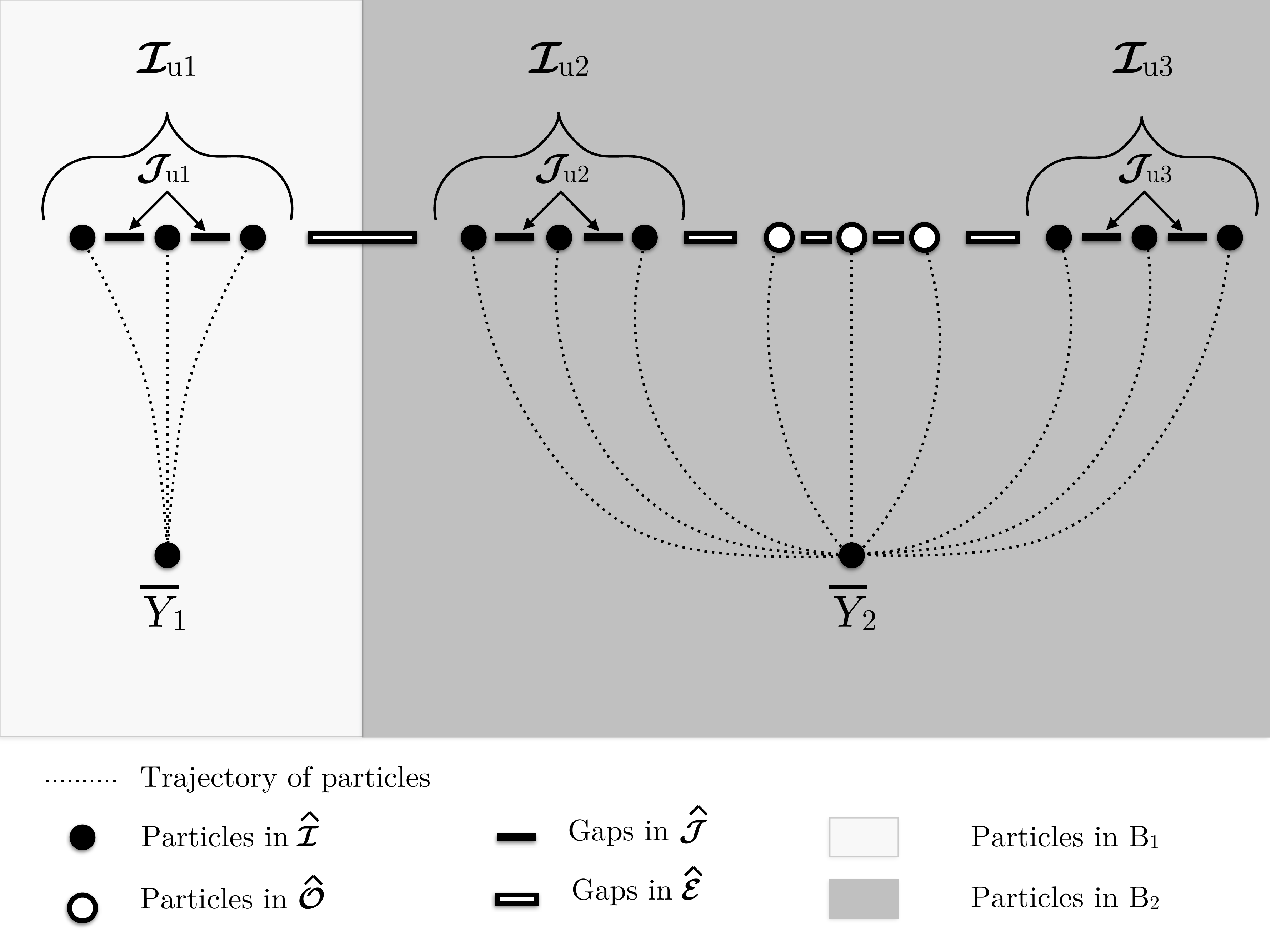}
\caption{Notations for the proof of Lemma \ref{Lem:strictconcave}. }
\label{fig:notations}
\end{figure}

We also introduce the partition of integers $[1,N]$ into sets of consecutive indices 
\[[1,N] = \bigcup_{l = 1}^q B_l\, , \quad (\forall l)\; (\forall i \in B_l)\quad \lim_{n\to +\infty} Y^n_i = \Yinf_l\, . \] 
Up to further extraction, we can assume that
\begin{equation}
(\forall i,j \in [1,N-1])\quad  \displaystyle \lim_{n\to +\infty} \frac{Y^n_{i+1}-Y^n_i}{Y^n_{j+1}-Y^n_j} = \gamma_{i,j} \in  \R_+ \cup \{ +\infty\}.
\end{equation}

Suppose $q=n$. Then for all $i< j$ , $\Yinf_{i} < \Yinf_{j}$. Hence, we can pass directly to the limit in $\H^N_{m+1}$ and $\F^N_m$.  
We obtain the existence of some $\Yinf$ such that  $|\Yinf|=1$, and more importantly,
\[\F^N_m(\Yinf)\geq 0\quad \text{and} \quad \H^N_{m+1}(\Yinf)=0\,.\] 
We deduce from the equality case in Cauchy-Schwarz inequality \eqref{eq:cauchy-schwarz H}  that there exists $\alpha \in \R$ such that 
\begin{equation} \label{eq:pointcritique}
 \nabla \F^N_m\left( \Yinf\right)+\alpha \Yinf=0.
\end{equation} 
Thus $\Yinf$ is a critical point of $\F^N_{m,\alpha}$. Taking the scalar product of \eqref{eq:pointcritique} with $\Yinf$, we get 
\[\alpha = (m-1) \F^N_m(\Yinf) \geq 0\,.\] 
If $\chi>C_N$ this is in contradiction with Corollary \ref{cro:nocritique} (take $p=N$ there). 

If $\chi=C_N$, Corollary \ref{cro:nocritique} implies $\alpha=0$, and back to equation \eqref{eq:pointcritique}, we see that $\Yinf$ is the unique critical point of $\F^N_m$ on the sphere. 


The case $q < N$ is more delicate to handle with. To this end, we isolate the fastest dynamics during blow up. 
We introduce $\hat \In$ the set of particles that blow-up with the fastest rate. It is characterized by those relative distances which are comparable to any other relative distances (including a much faster collapse), 
\[\hat \Inl =\left\{ i  \in [1,N-1] \left| \forall j\in [1,N-1], \gamma_{i,j}<+\infty \right. \right\} \quad \mbox{and}\quad \hat\In=\bigcup_{i \in \hat\Inl}[i,i+1]\, .  \]
The set $\hat \Inl$ is not empty. Otherwise, there would exist an application $S:[1,N-1]\to[1,N-1]$ such that $\gamma_{i,S(i)} = +\infty$, and accordingly $\gamma_{S(i),i} = 0$. Then, there would exist $i_0$  and $K$, such that $S^K(i_0) = i_0$. We would obtain a contradiction, since
\[
1=\gamma_{S^{K}(i_0),i_0} = \prod^{K}_{k=1}  \gamma_{S^{k}(i_0),S^{k-1}(i_0)} =0\, .
\]
We decompose $\hat\Inl$ into connected components $(\Inl_u)$,
\[ \hat\Inl= \bigcup_{u} \Inl_u\, , \] 
and we define $\In_u$ accordingly. The sets  $\In_u$ are two by two disjoint, and each $\In_u$ belongs to some part $B_l$. Moreover, each $\In_u$, made of consecutive indices, is a relative blow-up set. See Figure \ref{fig:notations} for the notations.
We denote $p_u= | \In_u |$. It is remarkable that we will prove eventually that all $p_u$ are equals.

The following properties are useful,
\begin{align}\label{eq:compdesdyn}
&\forall (i,j)\in \hat\Inl \times \left( [1,N-1]\setminus \hat\Inl \right)\quad  \gamma_{i,j}=0,\\ \label{eq:compdesdyn2}
&\forall  (i,j)\in \hat \Inl \times \hat \Inl \quad 0 < \gamma_{i,j} < +\infty\, . \end{align}
Indeed, let $i\in \hat \Inl$, and $j\in [1,N-1]$. Then either $\gamma_{i,j}=0$, or $0 < \gamma_{i,j} < +\infty$. If  $j\in \hat \Inl$, then $\gamma_{j,i}<+\infty$ implies that $\gamma_{i,j}>0$. On the other hand, if $j\notin \hat \Inl$, then there exists $k\in [1,N-1]$ such that $\gamma_{k,j} = 0$. By transitivity, $\gamma_{i,j}=\gamma_{i,k}\gamma_{k,j}  = 0$.  

According to Definition \ref{blowup}, we introduce the following notations:  $\hat \Inb =\bigcup  \Inb_u$, $\Out_u=[1,N]\setminus \In_u$, $\hat \Out =[1,N] \setminus  \hat\In$ and $\hat  \Outl =[1,N] \setminus  \hat\Inl$ . We also introduce the euclidean norm of relative distances inside $\hat \In$ (resp. $\In_u$),
\begin{equation}\label{variance}
\hat \Pi(Y) =\sqrt{\underset{i \in \hat\Inl}{\sum}\left( Y_{i+1} -Y_i \right)^2}\, , \quad \Pi_{u}= \sqrt{\underset{i  \in \Inl_u}{\sum}\left( Y_{i+1} -Y_i \right)^2}
\end{equation}
According to \eqref{eq:compdesdyn2}, for each $u$,  $\Pi_{u}$ and $\hat\Pi$ are  of the same order. We define the limiting ratio $r_u$ in accordance. Pick any $i\in \hat\Inl_u$, we have, 
\begin{equation}\label{eq:rapport}
1\geq \underset{n \rightarrow +\infty}{\lim} \displaystyle \frac{\Pi_{u}(Y^n)}{\hat\Pi (Y^n)}=  \sqrt{\displaystyle \frac{\sum_{k \in \Inl_u}(\gamma_{k,i})^2}{\sum_{k\in \hat \Inl}(\gamma_{k,i})^2 }}= r_u >0.
\end{equation}
Since $q< N$, at least two particles are collapsing. Therefore,  
\begin{equation}\label{eq:limite}
\underset{n \rightarrow +\infty}{\lim} \hat\Pi (Y^n) =0\,.
\end{equation}
As a consequence,  
\begin{equation}\label{eq:limite2}
(\forall u)\quad  \underset{n \rightarrow +\infty}{\lim} \Pi_{u}(Y^n) =0\, .
\end{equation}
This shows that each $\In_u$ is indeed a relative blow up set for $(Y^n)$. Indeed, the second condition of \eqref{strongblowupr} holds true by the very definition of $\hat\In$. On the other hand,  \eqref{eq:limite2} implies the first condition of \eqref{strongblowupr}, and \eqref{eq:compdesdyn} implies the third one.

Let $ \hat\gamma^n =    \underset{ j  \in \hat\Outl }{\max} \left( \frac{\hat\Pi (Y^n)}{Y^n_{j+1}-Y^n_j}  \right)$. According to \eqref{eq:compdesdyn}, $ \hat\gamma^n \to 0$ as $n\to +\infty$. Moreover,
\begin{align}\label{pluspetit}
(\forall u)\quad  \forall  i,j \in \In_u \times \Out_u, \quad &\displaystyle \frac{\hat\Pi (Y^n)}{Y^n_{j}-Y^n_i} \leq \hat\gamma^n,\\ \label{pluspetit2}
\forall  i\neq j \in [1,N]\times  \hat \Out, \quad &\displaystyle \left| \frac{\hat\Pi (Y^n)}{Y^n_{j}-Y^n_i} \right| \leq \hat\gamma^n.
\end{align}
Indeed, for \eqref{pluspetit} 
there exists $k \in \Inb_u$ such that 
$$\displaystyle \frac{\hat\Pi (Y^n)}{Y^n_{j}-Y^n_i} \leq \displaystyle \frac{\hat\Pi (Y^n)}{Y^n_{k+1}-Y^n_{k}} \leq \hat\gamma^n.$$
For \eqref{pluspetit2}, notice $j \in  \hat\Out$ implies  that both $j$ and $j-1$ belongs to $\hat\Outl$. Therefore,  
$$\left| \displaystyle \frac{\hat\Pi (Y^n)}{Y^n_{j}-Y^n_i} \right| \leq \displaystyle \max \left[ \frac{\hat\Pi (Y^n)}{Y^n_{j+1}-Y^n_{j}}, \frac{\hat\Pi (Y^n)}{Y^n_{j}-Y^n_{j-1}} \right] \leq \hat\gamma^n.$$
Using \eqref{eq:rapport}, \eqref{pluspetit} and \eqref{pluspetit2} we find 
\begin{align} \nonumber
\F^N_m \left(  \frac{Y^n}{\hat\Pi (Y^n)} \right) &= 
  \sum_{i=1}^{N-1} \left(  \frac{Y_{i+1}^n-Y_{i}^n}{\hat\Pi (Y^n)}  \right)^{1-m} - \chi \sum_{1\le i\neq j\le N}  \left| \frac{Y_{j}^n-Y_{i}^n}{\hat\Pi (Y^n)} \right|^{1-m}  \\ \nonumber
 &=\sum_u \left[\sum_{i \in \Inl_u} \left(  \frac{Y_{i+1}^n-Y_{i}^n}{\hat\Pi (Y^n)}  \right)^{1-m} - \chi \sum_{ i,i' \in \In_u }  \left| \frac{Y_{i'}^n-Y_{i}^n}{\hat\Pi (Y^n)} \right|^{1-m} \right]\\  \nonumber 
  &+\sum_{j  \in  \hat\Outl } \left(  \frac{Y_{j+1}^n-Y_{j}^n}{\hat\Pi (Y^n)}  \right)^{1-m} - \chi \sum_{ i\neq j \in  \hat\Out^2 }  \left| \frac{Y_{j}^n-Y_{i}^n}{\hat\Pi (Y^n)} \right|^{1-m} \\ \nonumber
 &-2 \chi \sum_u\sum_{ i \in \In_u, j \in \Out_u  }  \left| \frac{Y_{j}^n-Y_{i}^n}{\hat\Pi (Y^n)} \right|^{1-m} \\   \nonumber
&=  \sum_u \left[\left(\frac{\Pi_{u}(Y^n)}{\hat\Pi (Y^n)}\right)^{1-m}  \F^{p_u}_m \left(  \frac{Y_{\In_u}^n}{\Pi_{u}(Y^n)} \right)\right] + O\left(\left[ N+\chi N^2\right] \left(  \hat\gamma^n\right)^{m-1}\right) \\ 
\label{energie} &=\sum_u (r_u)^{1-m}  \F^{p_u}_m  \left(  Z_u \right)+ \epsilon(\frac{1}{n}),
 \end{align}
where $\epsilon(\frac{1}{n}) \underset{n \rightarrow +\infty}{\longrightarrow} 0$.
Similarly we find, for all $j \in \hat\Out $ 
\begin{align}\nonumber
 \left| \nabla_j \F^N_m\left(  \frac{Y^n }{\hat\Pi (Y^n)}\right)\right| &\leq (m-1) \left| \left(-{\left( \frac{Y_{j+1}^n -Y_{j}^n}{\hat\Pi (Y^n)}\right)^{-m}} + \left( \frac{Y_{j}^n -Y_{j-1}^n}{\hat\Pi (Y^n)}\right)^{-m}\right) \right|\\ \nonumber
&+(m-1) \chi  \sum_{k \neq j } {\left|\frac{Y_{k}^n -Y_{j}^n}{\hat\Pi (Y^n)} \right|^{-m}} \\
\label{gradext} &\leq 2(m-1) \left(   \hat\gamma^n \right)^{m} +(m-1) \chi  \sum_{k \neq j } {\left|\frac{Y_{j+1}^n -Y_{j}^n}{\hat\Pi (Y^n)} \right|^{-m}} \\ \nonumber
&\leq (m-1)\left(2+\chi N \right) \left(   \hat\gamma^n \right)^{m} \underset{n \rightarrow +\infty}{\longrightarrow} 0.
\end{align}
On the contrary, for all $u$ and $i \in \In_u$, we find
\begin{align}
\nonumber  \nabla_i \F^N_m\left(  \frac{Y^n }{\hat\Pi (Y^n)}\right) &= \nabla_i \F^{p_u}_m\left(  \frac{Y_{\In_u}^n }{\hat\Pi (Y^n)}  \right)
 +\sum_{j \in \Inb_u } (\delta_{i,j+1}-\delta_{i,j}) {\left(\frac{Y_{j+1}^n -Y_{j}^n}{\hat\Pi (Y^n)} \right)^{-m}}   \\ \nonumber 
&+(m-1) \chi  \sum_{j\in \Out_u } \sign(j-i){\left(\frac{Y_{j}^n -Y_{i}^n}{\hat\Pi (Y^n)} \right)^{-m}} \\
\nonumber &= \left(  \frac{\Pi_{u}(Y^n)}{\hat\Pi (Y^n)}  \right)^{-m} \nabla_i \F^{p_u}_m\left(  \frac{Y_{\In_u}^n }{\Pi_{u}(Y^n)}  \right) + O\left((m-1)\left(2+\chi N \right) \left(   \hat\gamma^n \right)^{m} \right)\\ \nonumber
&= (r_u)^{-m} \nabla_i \F^{p_u}_m\left( Z_u  \right) +\epsilon(\frac{1}{n}).
\end{align}
Therefore
\begin{align}
 \nonumber \left| \nabla \F^N_m\left(  \frac{Y^n}{\hat\Pi (Y^n)}  \right) \right|^2 &= \sum_u\sum_{i\in \In_u}   \left| \nabla_i \F^N_m\left(  \frac{Y^n }{\hat\Pi (Y^n)}  \right)\right|^2     + \sum_{j\in \hat\Out}   \left| \nabla_j \F^N_m\left(  \frac{Y^n }{\hat\Pi (Y^n)}  \right)\right|^2  \\
 \label{gradint} &=   \sum_u (r_u)^{-2m} \left| \nabla \F^{p_u}_m\left(  Z_u  \right)\right|^2 +\epsilon(\frac{1}{n}).
\end{align}
Finally, \eqref{energie}, \eqref{gradint} and \eqref{gradext} together yields 
\begin{align}
\nonumber \H^N_{m+1}(Y^n)&= \left( \hat\Pi (Y^n) \right)^{-2m}\left| \nabla \F^N_m\left(  \frac{Y^n}{\hat\Pi (Y^n)}  \right) \right|^2 \\
\nonumber &+  \left( \hat\Pi (Y^n) \right)^{2-2m} \left((m-1)\F^N_m\left(  \frac{Y^n}{\hat\Pi (Y^n)}  \right)\right)^2 \\
\nonumber  &=  \left( \hat\Pi (Y^n) \right)^{-2m}\left[ \sum_u (r_u)^{-2m} \left| \nabla \F^{p_u}_m\left( Z_u \right) \right|^2\right. \\
\nonumber  &+ \left. \left( \hat\Pi (Y^n) \right)^{2} \left((m-1) \sum_u (r_u)^{1-m}  \F^p_m  \left(  Z_u \right) \right)^2 + \epsilon( \frac{1}{n})\right] \\
 \label{lim} &= \left( \hat\Pi (Y^n) \right)^{-2m}\left[ \sum_u (r_u)^{-2m} \left| \nabla \F^{p_u}_m\left( Z_u \right) \right|^2+\epsilon( \frac{1}{n})\right] 
\end{align}
The limit $n\rightarrow +\infty$ in \eqref{lim} implies that for all $u$, $Z_u$ is a critical point for $\F^{p_u}_m$. This is a contradiction with Corollary \ref{cro:nocritique} unless, for any $u$, $\chi=C_{p_u}=C_p$. 
In the latter case, each relative blow up set must have the same number of particles $p$, which is the critical number of particles, and $Z_u$ is a critical point of $\F^p_m$. In particular $\F^p_m(Z_u)=0$. Remark that $q < N$ implies $p < N$.
\end{proof}

\begin{Rk}
To obtain some rigidity results about the structure of the blow-up set, as in \cite{KS1Dp}, it is necessary to understand the next order term in \eqref{lim}.
\end{Rk}

We are now in position to prove our first main theorem.
\begin{proof}[Proof of Theorem \ref{thm:explosionintro}]
(1) Assume by contradiction that blow-up does not occur. This implies that $f_{m+1}(t)\geq 0$ for all $t$, otherwise blow-up necessary occurs in finite time.
We first establish that the second moment is uniformly bounded.
According to \eqref{dynmomentdeux} and Lemma \ref{Lem:strictconcave}, if there exists no $p \in [1,N-1]$ such that $\chi = C_p$, then there exists $\delta>0$ such that  
\begin{equation}\label{etapeinutile}
\frac{d^2 f_{m+1}}{dt^2}(t) \leq - \delta \left(f_{m+1}\left(t\right)\right)^{-1}  \,. 
\end{equation}
Multiplying this differential inequality by $\frac{d f_{m+1}}{dt}(t)\geq 0 $ \eqref{dynmomentun}, and integrating over $(0,t)$, we deduce 
$$\frac12\left(\frac{d f_{m+1}}{dt}(t)\right)^2 + \delta   \ln \left( {f_{m+1}\left(t\right)} \right)  \leq C_0\,, $$
where the constant $C_0$ depends on the initial data.  
It implies that ${f_{m+1}\left(t\right)}$ is uniformly bounded by $\exp(C_0/\delta)$. In particular, back to \eqref{etapeinutile} we realize that the function ${f_{m+1}}$ is 
$\bar \delta$ concave, for some $\bar \delta >0$. Consequently, $f_{m+1}$ cannot remain nonnegative for all time. This is a contradiction.

(2) If there exists  $p \in [1,N-1]$ such that $\chi = C_p$, Lemma \ref{Lem:strictconcave} provides us with the following alternative: either $\H$ is uniformly bounded below along the trajectory, hence blow-up occurs in finite time; or the renormalized locations $Y=\frac{X}{|X|}$ blow up in infinite time and there exists a relative blow up set containing $p$ particles and converging towards the critical point of $\F^p_m$ on the unit sphere, up to extraction of some subsequence.

\section{Dichotomy and discussion}\label{sec:conc}

To complete the results of this paper, we establish that the minimal number of particles taking part in each blow-up set is at least $k+1$, where $k$ is defined such as $\chi < C_k$. As a consequence, the solution exists for all time in the subcritical regime. In the super critical regime, the next step of investigation would be to describe in a refined way the blow up mechanism in the spirit of \cite{KS1Dp}. For instance, one may ask what is the mass contained in the first singularity, or whether the energy must become negative before the occurrence of blow up. A first step towards these issue is given by the following propositions. The refined description of blow-up, similar to \cite{KS1Dp}, is left for future work. Also left for future work is the asymptotics of renormalized solutions in the subcritical regime. In the logarithmic case, it was proven in \cite{BCC08} that solutions converge exponentially fast to a unique equilibrium state after suitable parabolic rescaling (see \cite{CC12,Campos-Dolbeault} for a similar result in the continuous case). However, the argument of \cite{BCC08} cannot be readily extended to the present case because the nonzero homogeneity creates some additional terms that cannot be handled easily (see \cite{Franca} for similar issues in the continuous setting).

\begin{Def}[Blow-up of particles]\label{weakblowup}
Let $T \in [0,+\infty)$ and $X \in C^0\left([0,T),\Rord^N\right)$. Let $\In_w =[l,r]  \subset [1,N]$ be a set of consecutive indices,
$\Out_w =[1,N] \setminus  \In$, $\Inl_w= [l,r-1] $ and  $\Inb_w = \left( \{l-1\} \cup \{ r \}\right) \cap [1,N-1]$. We say that $\In$  weakly blows up at time T if
\begin{equation}\label{weakblowup}
\forall i \in \Inl_w  \quad \liminf_{t\to T^- } \left( X_{i+1}-X_i \right) =0.
\end{equation}
when the set  $\In$ is maximal for the inclusion, 
we refer to it as a weak blow-up set.
\end{Def}

\begin{Rk}
For a solution $X$ of the the discrete gradient flow \eqref{flotgradientdiscretexplicite}--\eqref{flotgradientdiscretexplicite-boundary}, either $X$ exists for all time or there exists a weak blow-up set for $X$.
\end{Rk}

\begin{Prop}\label{blowupaumoinsk}.
Let $X$ be a solution of the the discrete gradient flow \eqref{flotgradientdiscretexplicite}--\eqref{flotgradientdiscretexplicite-boundary} and $\chi < C_k$. Any finite time weak blow up set contains at least $k+1$ particles.
\end{Prop}
A proof of an analogous result in the case of logarithmic homogeneity can be found in \cite{KS1Dp}. The strategy there was to localize the energy on specific subsets of particles, and deduce fruitful estimates. 
\begin{proof}
We consider a weak blow-up set $\In_w$ made of $p\leq k$ particles. 
Since $\In_w$ is a weak blow-up set, by maximality, we have
$$\min_{j\in \Inb_w} \left(  \liminf_{t\to T^- } \left(X_{j+1}-X_{j}\right)\right)>0.$$ Therefore there exists $c>0$ such that for any $j \in \Out_w$, $i\in \In_w$ and $s\in [0,T)$,
\begin{equation}\label{courteportee}
|X_j\left(s\right)-X_i\left(s\right)|\geq \frac1c.
\end{equation} 
Let us consider the local energy
\[
\F^p_{m} \left( X \right)=\sum_{i \in \Inl_w   } \left(X_{i+1}-X_i\right)^{1-m} -\chi \sum_{ (i, j)\in \In_w \times \In_w \setminus \lbrace i \rbrace}   \vert X_i-X_j \vert^{1-m} .
\]
Thanks to \eqref{courteportee} and the Young inequality we find $A>0$ such that
\begin{align*}
\frac{d}{dt} \F^p_{m}  &=-\left< \nabla \F^p_{m} ,(\nabla_i \F^N_{m} )_{i\in \In_w} \right>_{\In_w} \\
&=-\| \nabla \F^p_{m} \|^2_{\ell^2(\In_w)} \\
&\quad \quad  +\left< \nabla \F^p_{m} , -\dfrac{\delta_{i,l+p} }{(X_{l+p+1}-X_{l+p})^{m}} +  \dfrac{\delta_{i,l} }{(X_{l}-X_{l-1})^{m}}+2\chi \sum_{k\in \Out_w}\frac{1}{(X_k-X_i)^{m}} \right>_{\In_w} \\
&\leq -\| \nabla \F^p_{m} \|^2_{\ell^2(\In_w)} +\left(2+2\chi \right)\| \nabla \F^p_{m} \|_{\ell^2(\In_w)} \|\sum_{k\in \Out_w}\frac{1}{(X_k-X_i)^{m}} \|_{\ell^2(\In_w)}  \\
& \leq -\frac{1}{2} \|\nabla \F^p_{m} \|^2+A^2,
\end{align*}
and therefore for any $t>0$:
\begin{equation}\label{ekdecroissantalmost}
\F^p_{m}  \left( X(t) \right)\leq \F^p_{m}  \left( X(0) \right)+ tA^2.
\end{equation}
We define $\theta$ such that $C_p=\frac{\chi }{\theta}$, observe that $\theta=\frac{\chi}{C_p}<1$ since $p\leq k$.
Using the discrete HLS inequality given in Proposition \ref{thm:bornepardessous} and \eqref{ekdecroissantalmost} we obtain for any $t\in[0,T)$ and $i \in \Inl_w$:
\[
\left(X_{i+1}-X_i\right)^{1-m} \leq \sum_{i \in \Inl_w  } \left(X_{i+1}-X_i\right)^{1-m} \leq \frac{\F^p_{m} (0)+TA^2}{1-\theta}.
\]
\[
\left(X_{i+1}-X_i\right) \geq \left(\frac{1-\theta}{\F^p_{m} (0)-TA^2}  \right)^{\frac{1}{m-1}}.
\]
It is a contradiction with $\In_w$ being a weak blow-up set and proves Proposition \ref{blowupaumoinsk}.
\end{proof}
In the subcritical case this proposition proves Theorem \ref{thm:subcase}. In the super critical case it proves that a weak blow up set contains at least the critical number of particles. Using similar computations as  done in the proof of Lemma \ref{Lem:strictconcave}, one can obtained the following proposition.
\begin{Prop}\label{blowupaumoinskr}
Let $\chi < C_k$ and $X$ be a solution of the the discrete gradient flow \eqref{flotgradientdiscretexplicite}--\eqref{flotgradientdiscretexplicite-boundary} which blows up in finite time, then at least one relative blow up set has a nonpositive local energy and contains more than $k+1$ particles.
\end{Prop}
\begin{proof}
Let $Y^n$ be a subsequence of $\frac{X}{|X|}$. 
Following the proof of Lemma \ref{Lem:strictconcave} we can extract a subsequence and use equation \eqref{energie} that is:
\begin{align*}
\F^N_m(X_0) &\geq |X|^{1-m} \F^N_m(Y^n) =  |X|^{1-m} |\hat\Pi (Y^n)|^{1-m} \F^N_m \left(  \frac{Y^n}{\hat\Pi (Y^n)} \right) \\
&=  {\hat\Pi (Y^n)}^{1-m}\F^N_m \left(  \frac{Y^n}{\hat\Pi (Y^n)} \right) \\ &= |X|^{1-m} |\hat\Pi (Y^n)|^{1-m}\left(\sum_u (r_u)^{1-m} \F^p_m  \left(  Z_u \right)+ \epsilon(\frac{1}{n}) \right).
\end{align*}

If the second moment of $X$ is bounded we deduce
\begin{equation}
\sum_u (r_u)^{1-m}  \F^p_m  \left(  Z_u \right)\leq 0.
\end{equation}
In particular if $\chi=C_N$ let $V$ be the unique minimizer of $\F^N_m$ on the sphere. If moreover the second moment $X$ is bounded then $v=1$ and $\lim_{t \to \infty }Y(t)=V$.
\end{proof}

\end{proof}

To go further one need to prove for example an Harnack inequality for the weak blow up set, in the spirit of \cite{KS1Dp}. This question is left to a futur work, together with the study of the convergence towards equilibrium.

\section{Numerics}\label{sec:Numerics}

We present numerical simulations of the gradient flow system \eqref{flotgradientdiscretexplicite}-\eqref{flotgradientdiscretexplicite-boundary} which are driven by, and illustrate, previous analysis. Our main message  is that dynamics for $m>1$ are more nonlinear than the logarithmic case $m=1$. In particular, the evolution of the second momentum $|X(t)|^2$ is linear in the latter case, whereas it is genuinely concave in the case $m>1$.

In this section, we set $m = 1.2$, $N = 200$, and $\chi=1.45$ which is slightly above the critical value $\chi_m(N)$.

\begin{figure}
\includegraphics[width = 0.48\linewidth]{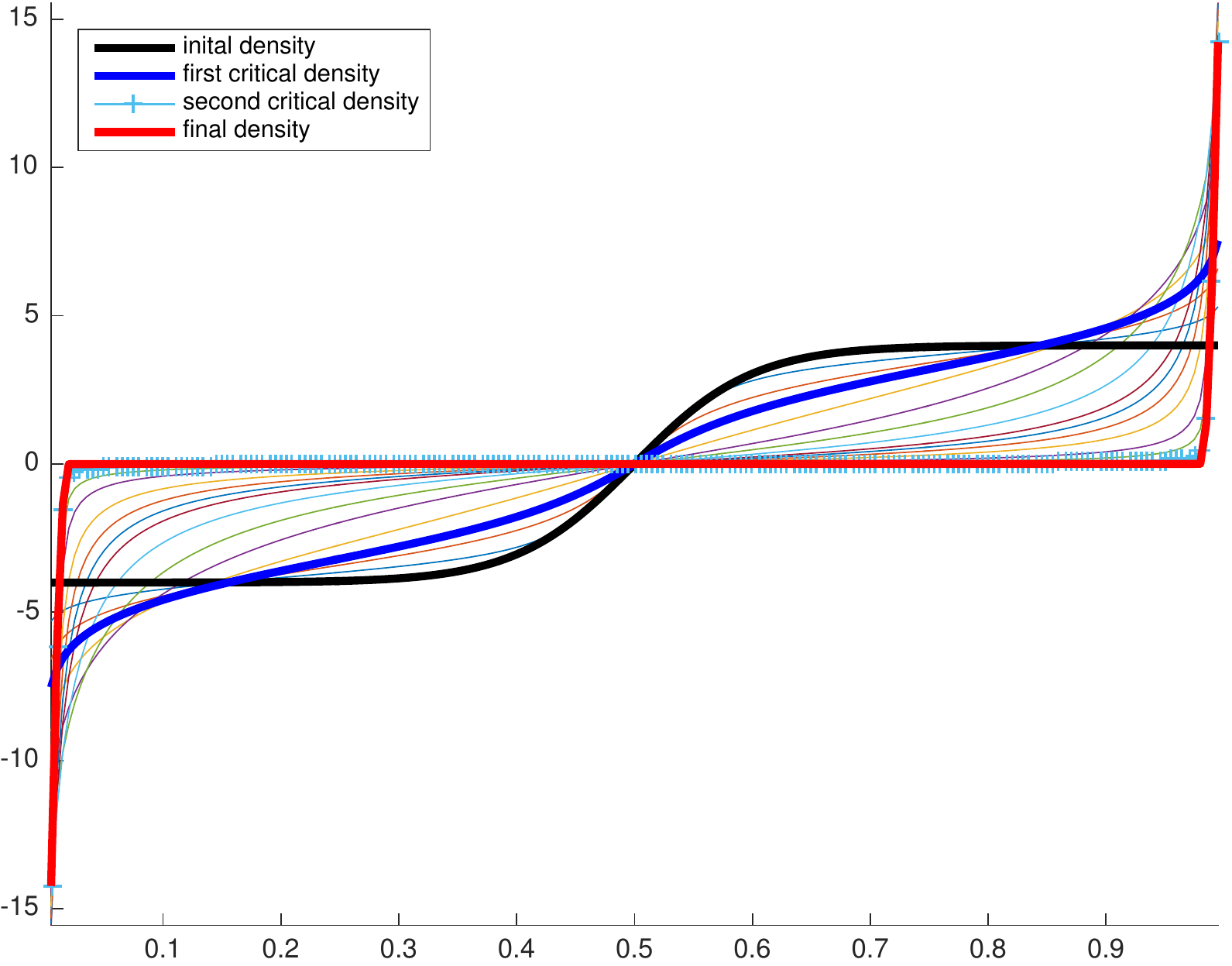}\quad
\includegraphics[width = 0.48\linewidth]{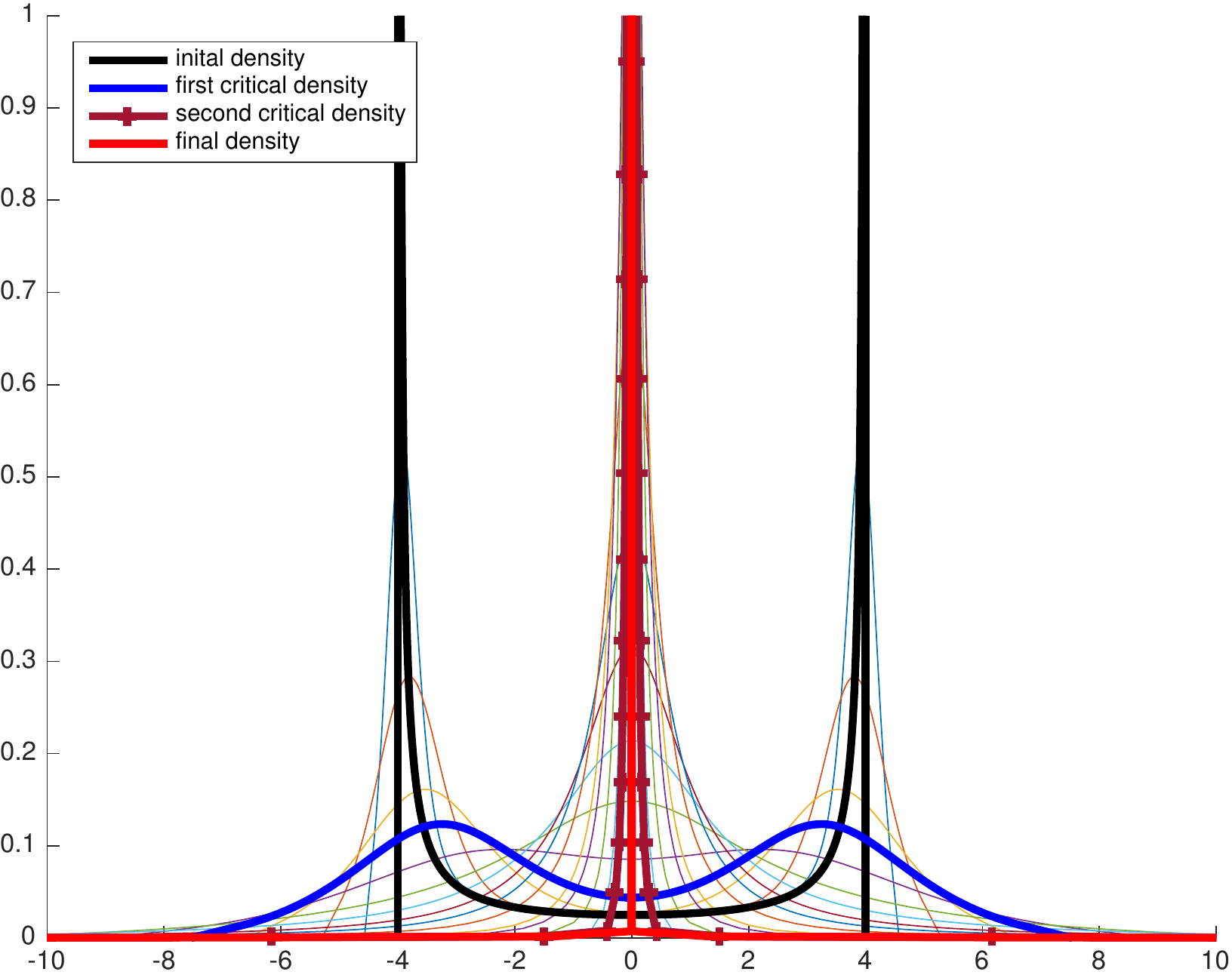}
\caption{Density and repartition functions.}
\label{fig:dense}
\end{figure}

\subsection{Methodology}
We used an Euler implicite scheme wich involves Newton's algorithm at each time step to compute $X(t_{n+1})$ from $X(t_n)$. This is equivalent to a steepest-descent scheme, also known as the Jordan-Kinderlehrer-Otto scheme in our context \cite{JKO,GosTos06,BCC08}.

The initial state is a regular discretization of $X_0(p)=4*\tanh\left(10*\left(p-0.5\right)\right)$ (see Fig. \ref{fig:dense}(left)). It corresponds to an initial density which is equally concentrated around two peaks (see Fig. \ref{fig:dense}(right)). The two peaks are chosen sufficiently far apart, so that the initial energy is positive, $\F(X_0)>0$.

In order to obtain better accuracy in some interesting time ranges, we used a non uniform time step (details given below).
\begin{figure}
\includegraphics[width = 0.48\linewidth]{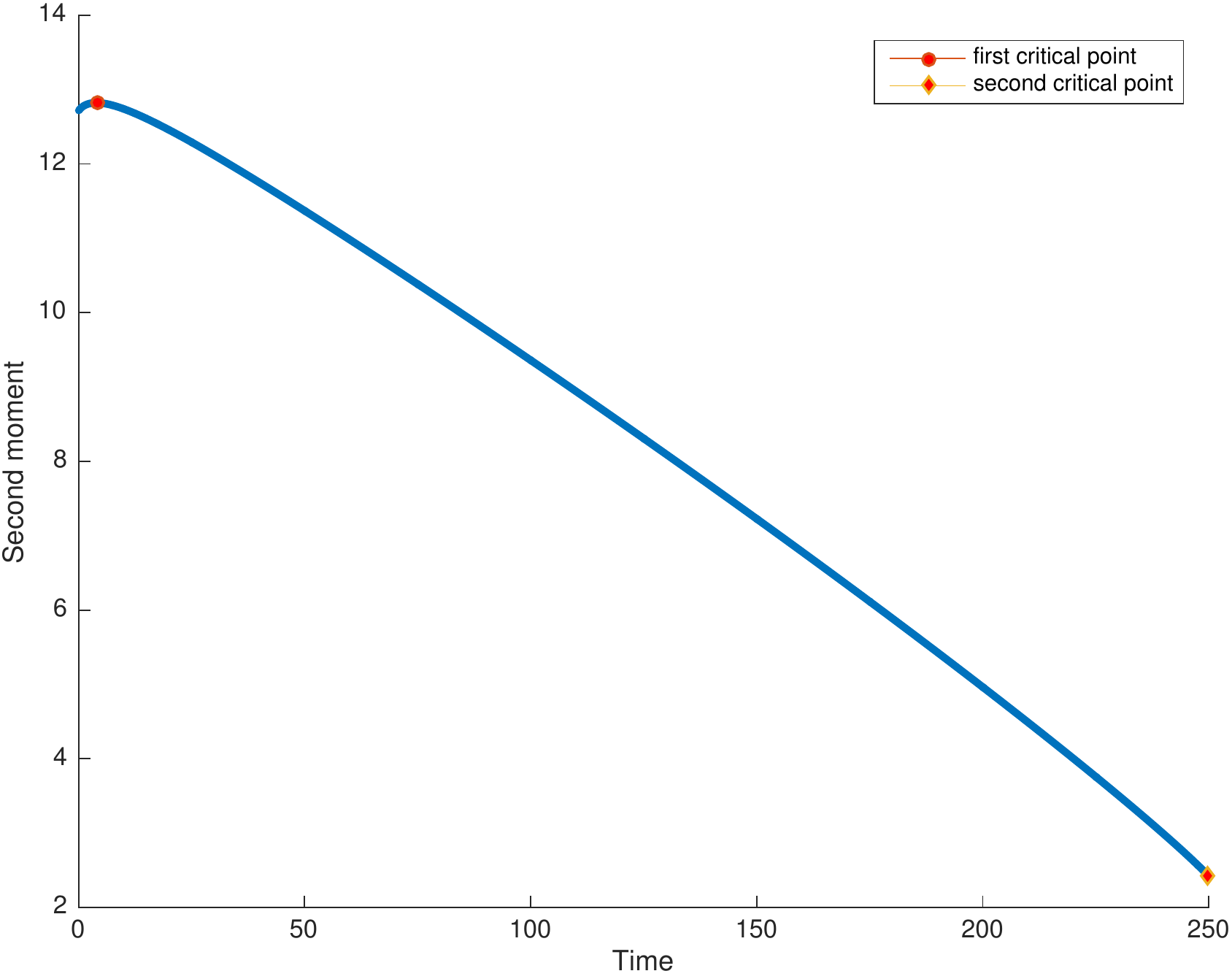}\quad
\includegraphics[width = 0.48\linewidth]{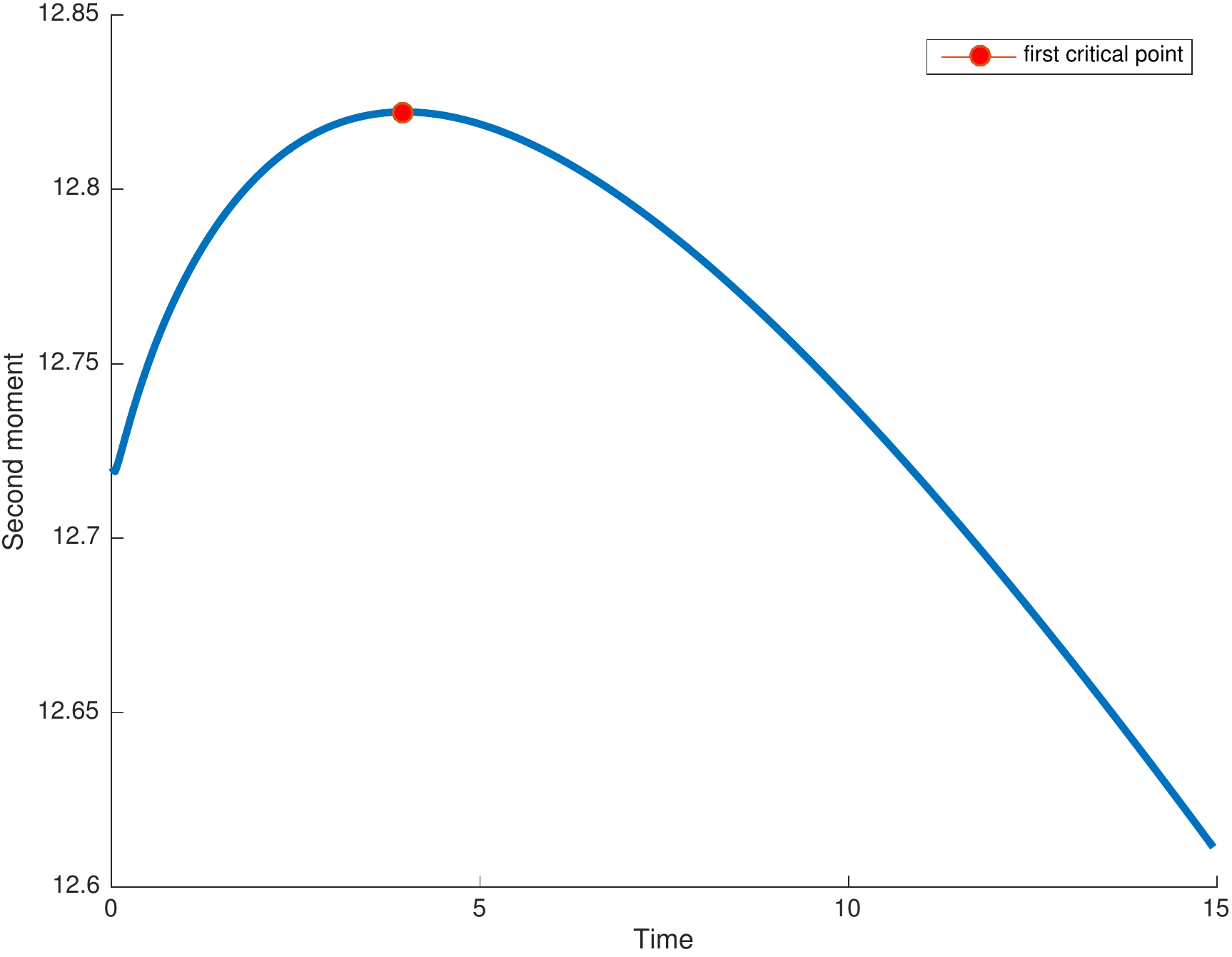}\\
\caption{(Left) Time evolution of the second moment $|X(t)|^2$ accross the time range of simulation. Clearly, it is a genuinely concave function, as opposed to the logarithmic case $m = 1$. (Right) Zoom on the initial phase. The red dot indicates the maximum of the second moment, wich coincides with the time at which energy changes sign \eqref{eq:1st derivative X^2}. The second red dot in the left plot is a marker of the blow-up phase (see Fig. \ref{fig:freeenergy}).}
\label{fig:secondmoment}
\end{figure}

\begin{figure}
\includegraphics[width = 0.48\linewidth]{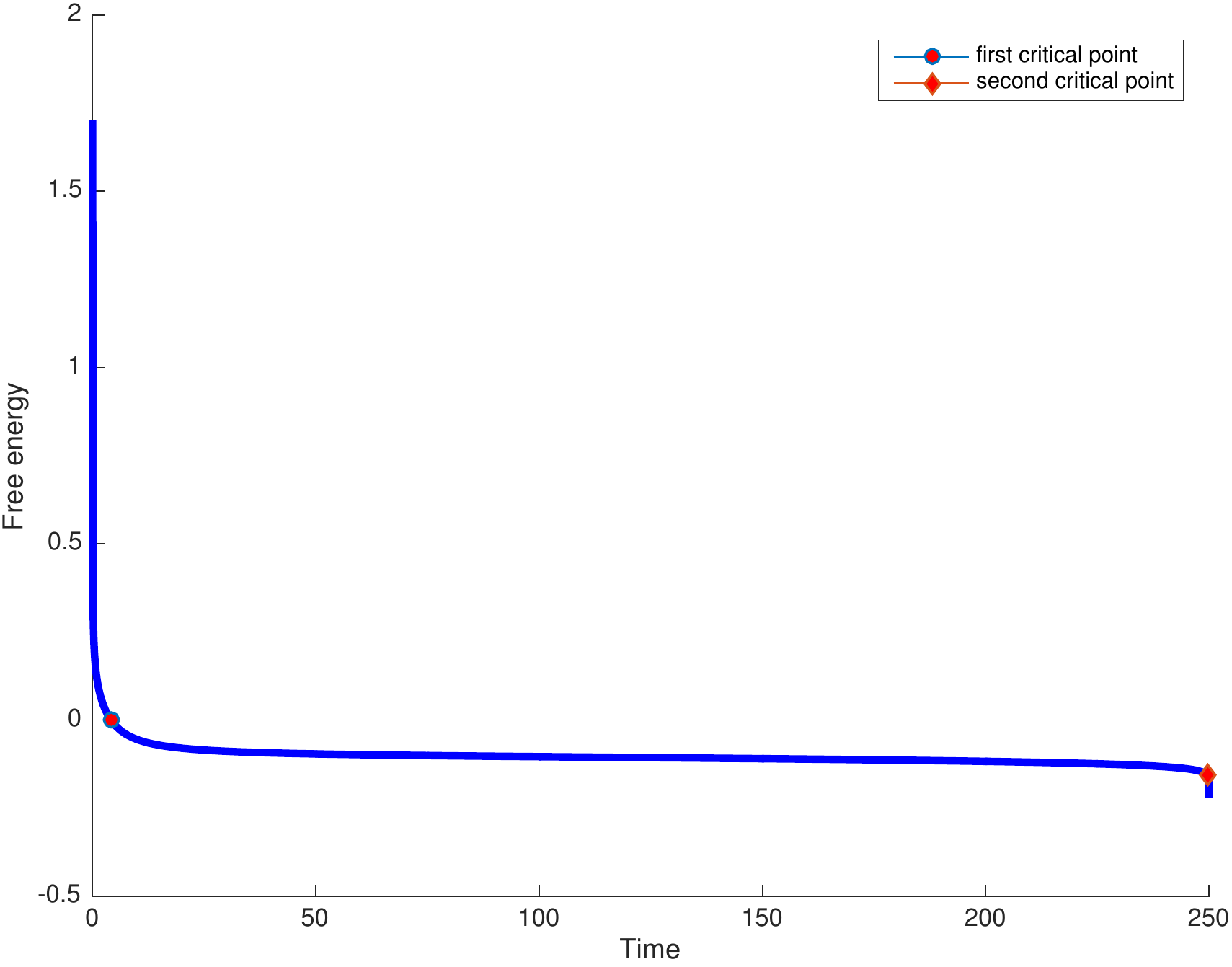}\quad
\includegraphics[width = 0.48\linewidth]{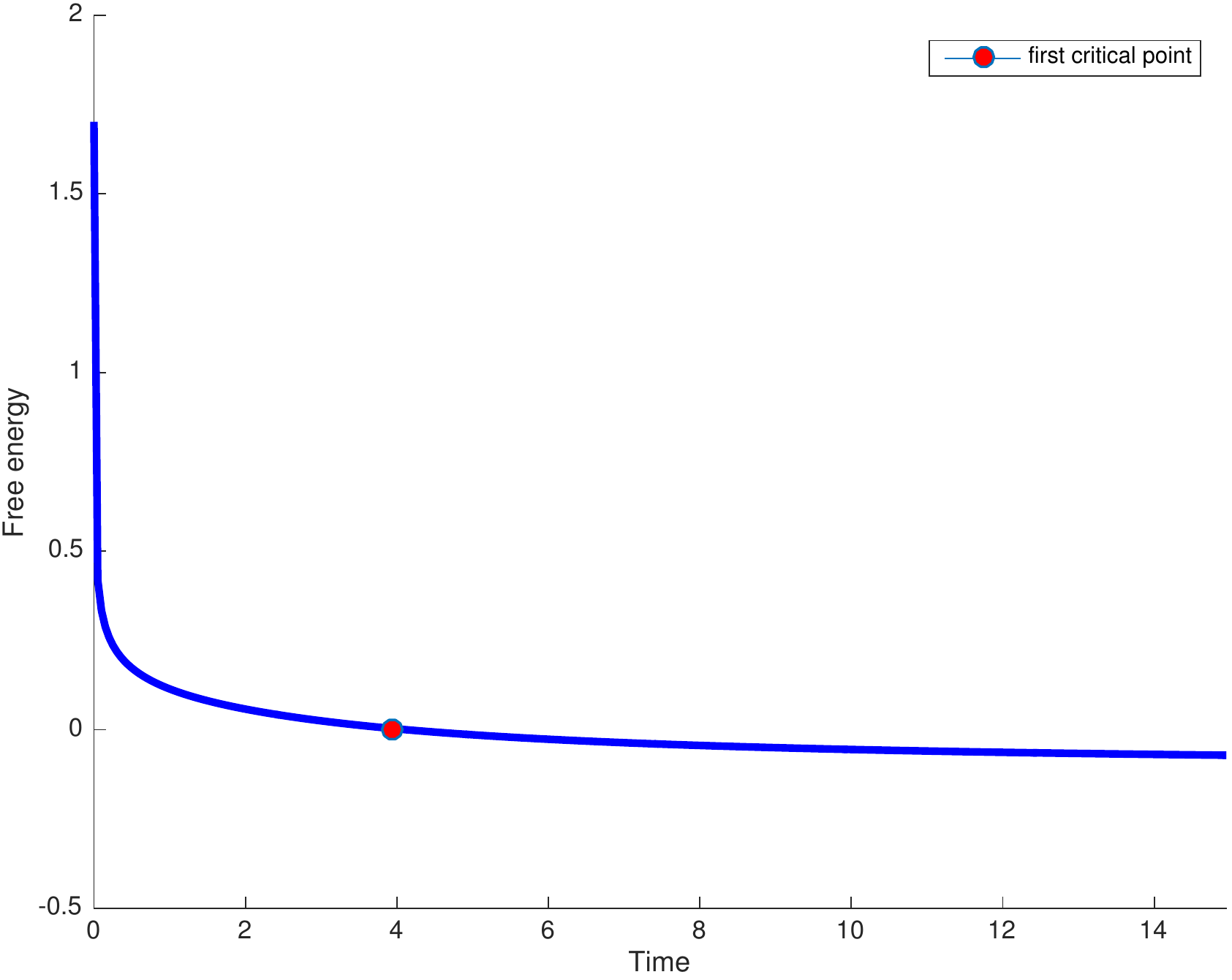}\\
\includegraphics[width = 0.48\linewidth]{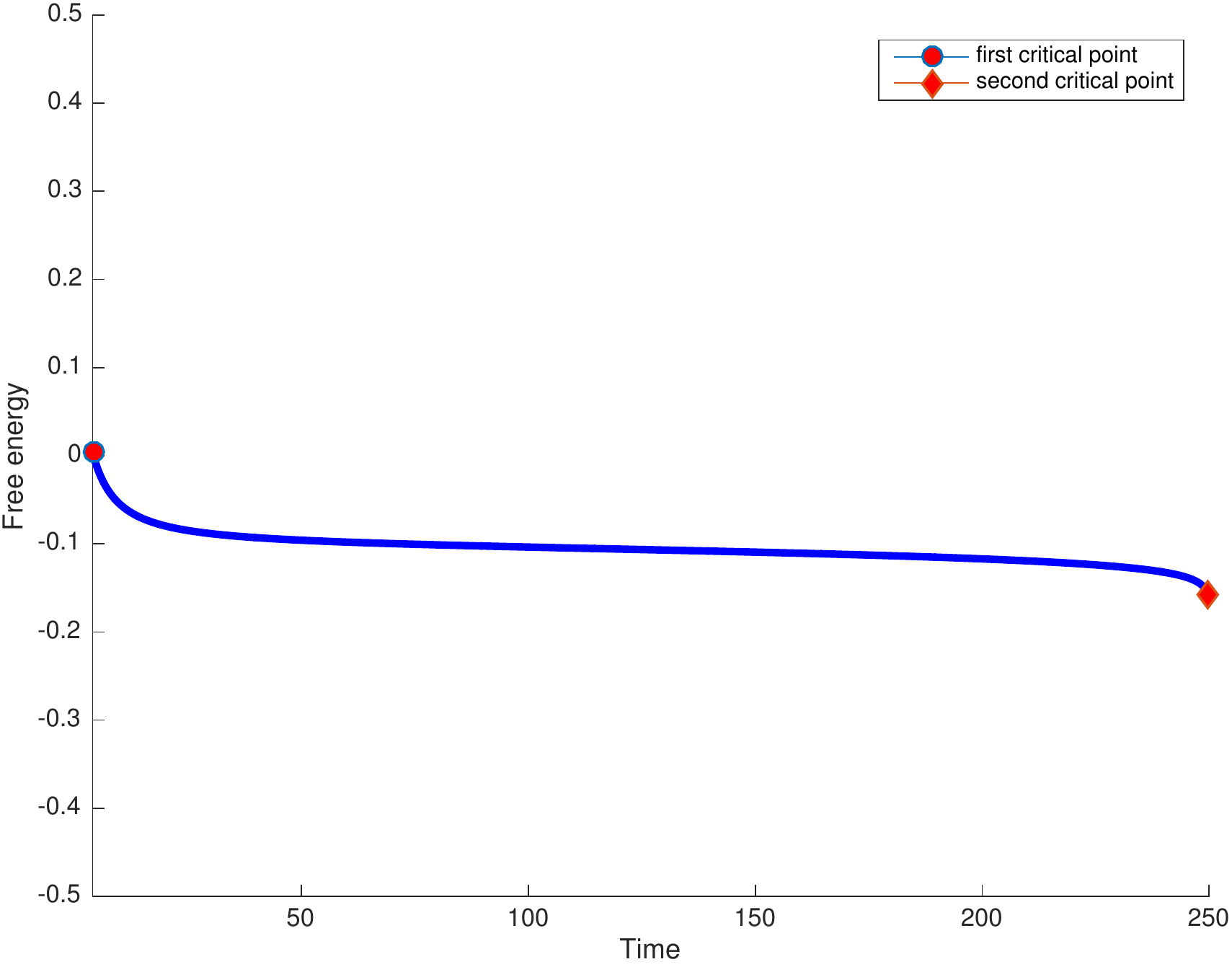} \quad
\includegraphics[width = 0.48\linewidth]{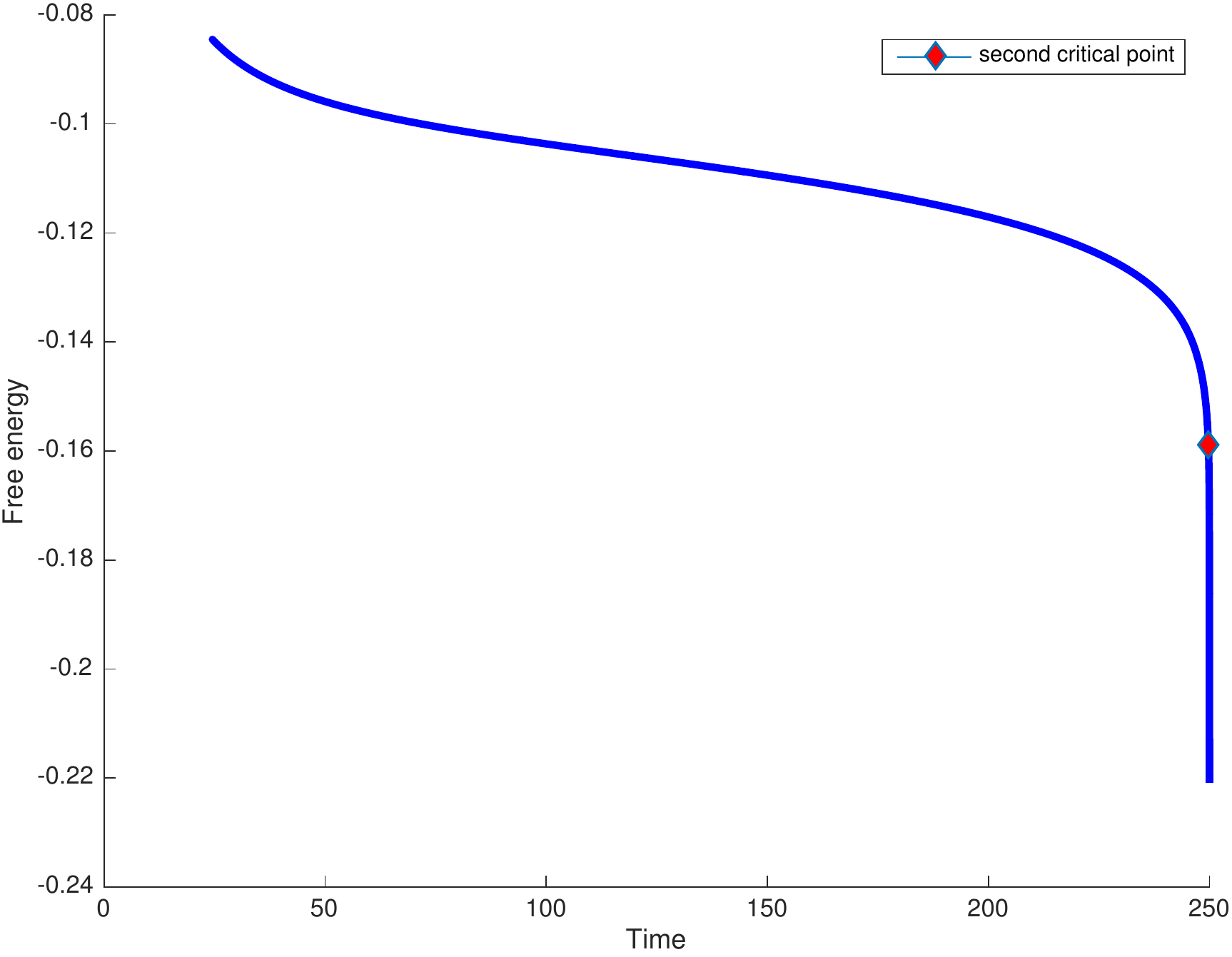}
\caption{Time evolution of the free energy $\F_m^N(X(t))$ (global evolution plus zooms on three phases).}
\label{fig:freeenergy}
\end{figure}

\subsection{Results}
Firstly, we confirm that the second moment is genuinely concave \eqref{eq:2nd derivative X^2}, see Fig. \ref{fig:secondmoment}. Moreover, it changes monotonicity after some time $t_0\approx 4$, in accordance with analysis developped in Section \ref{sec:BU}. Consequently, the solution blows-up in finite time.


Nonlinear effects are more visible on the evolution of the energy functional along the trajectory $\F(X(t))$. We distinguish between three phases. In a first stage, the energy decays rapidly from an initial positive value, to negative values. This corresponds to rapid smoothing and rearrangement of the initial density (Fig. \ref{fig:dense}(right)). During this first phase, the second moment grows, indicating a global spreading of particles  (Fig. \ref{fig:dense}(right)). In a second stage, the second moment decays almost linearly (Fig. \ref{fig:secondmoment}), indicating the global shrinkage of the cloud of particles. However, this happens at an approximately constant energy level (Fig. \ref{fig:freeenergy}). After a relatively long phase of rearrangement at constant energy, the energy suddenly drops down, indicating aggregation of a subset of particles into a blow-up set.

As mentioned above, time steps were not chosen uniformly. More precisely, $\Delta t = 5.10^{-2}$ up to $t_0=4$, corresponding to the initial phase of growth of the second moment, and dramatic decrease of the energy.
Then, time step was increased to $\Delta t = 5.10^{-1}$, during the long phase of rearrangement at almost constant energy, up to the time where the Newton's algorithm did not converge to a solution of the implicit Euler scheme (here, $t_1 = 249.75$). Lastly, the time step was decreased in an adaptive manneer in order to continue iterations as close as possible to the blow-up time.

\bibliographystyle{abbrv}

\end{document}